\documentclass{amsart}
\usepackage{amsmath,amssymb,amsxtra,anysize,tikz,hyperref,tikz-cd,combelow}
\usepackage{enumitem}
\usepackage[scr]{rsfso}

\newtheorem{theorem}{Theorem}[section]

\newtheorem{corollary}[theorem]{Corollary} 
\newtheorem{lemma}[theorem]{Lemma} 
\newtheorem{proposition}[theorem]{Proposition} 
\newtheorem{example}[theorem]{Example}

\newtheorem{remark}[theorem]{Remark}  
\newtheorem{definition}[theorem]{Definition}

\newsavebox\negcrossing
\begin{lrbox}{\negcrossing}
\begin{tikzpicture}[baseline=2pt, scale=.4]
\draw[->, thick] (1,0)--(0,1);
\draw[thick] (0,0)--(.4,.4);
\draw[->, thick] (.6,.6)--(1,1);
\end{tikzpicture}
\end{lrbox}

\newsavebox\poscrossing
\begin{lrbox}{\poscrossing}
\begin{tikzpicture}[baseline=2pt, scale=.4]
\draw[thick] (1,0)--(.6,.4);
\draw[->, thick] (.4,.6)--(0,1);
\draw[->, thick] (0,0)--(1,1);
\end{tikzpicture}
\end{lrbox}

\newsavebox\zerosmoothing
\begin{lrbox}{\zerosmoothing}
\begin{tikzpicture}[baseline=2pt, scale=.4]
\draw (0,1) to [out=-45,in=-135](1,1);
\draw (0,0) to [out=45,in=135](1,0);
\end{tikzpicture}
\end{lrbox}

\newsavebox\onesmoothing
\begin{lrbox}{\onesmoothing}
\begin{tikzpicture}[baseline=2pt, scale=.4]
\draw[->,thick] (0,0) to [out=45,in=-45](0,1);
\draw[->,thick] (1,0) to [out=135,in=-135] (1,1);
\end{tikzpicture}
\end{lrbox}

\newsavebox\unknot
\begin{lrbox}{\unknot}
\begin{tikzpicture}[baseline=-3pt, scale=.4]
\draw[thick] (0,0) circle (14pt);
\end{tikzpicture}
\end{lrbox}

\title{Structures in HOMFLY-PT homology}
\author{Alex Chandler}
\email{achandler@math.ucdavis.edu}
\author{Eugene Gorsky}
\email{egorskiy@math.ucdavis.edu}

\address{Department of Mathematics, University of California, Davis}

\newcommand{\CS}{\mathcal{C}}
\newcommand{\HH}{\mathrm{HH}}
\newcommand{\HHH}{\mathcal{H}}
\newcommand{\CP}{\mathscr{P}}
\newcommand{\DGR}{\textnormal{DGR}}
\newcommand{\sll}{\mathfrak{sl}}
\newcommand{\C}{\mathbb{C}}
\newcommand{\Q}{\mathbb{Q}}
\newcommand{\CY}{\mathcal{Y}}
\newcommand{\HY}{\mathrm{HY}}
\newcommand{\ad}{\mathrm{ad}}
\newcommand{\End}{\mathrm{End}}
\newcommand{\Ker}{\mathrm{Ker}}
\newcommand{\Imm}{\mathrm{Im}}

\definecolor{grayy}{rgb}{0.53, 0.23, 0.23}

\begin{document}

\maketitle

\begin{abstract}
We study the structure of triply graded Khovanov-Rozansky homology using both the data recently computed by Nakagane and Sano for knots up to 11 crossings, and the $\sll(2)$ action defined by the second author, Hogancamp and Mellit. In particular, we compute the HOMFLY-PT $S$-invariant for all knots in the dataset, and compare it to the $\sll(N)$ concordance invariants.
\end{abstract}

\section{Introduction}

In this note, we study the structure of triply graded Khovanov-Rozansky homology \cite{KR2,KhSoergel} which categorifies  HOMFLY-PT polynomial, and its interactions with $\sll(N)$ Khovanov-Rozansky homology \cite{KR1}. Although this link homology theory was defined over 15 years ago, its definition is rather involved, and the progress in computing and understanding this homology has been rather slow. 

Recently, triply graded homology was computed for all torus knots in \cite{EH,Hog,Mellit, HM}. In a different direction, Nakagane and Sano \cite{NS} wrote a program which computed triply graded homology for all knots with at most 10 crossings and most 11-crossing knots. 
These developments indicate that the triply graded homology can be used as a rich source of numerical data about knot invariants which is readily available in many examples. We will argue that for the Nakagane-Sano dataset lots of theoretical tools (such as spectral sequences) simplify dramatically, and can be computed from triply graded homology.   

The first, and the most basic problem in approaching this data is its visualization. 
We will exclusively work with {\em reduced} homology $\HHH(K)$ for knots which is known to be finite-dimensional. By a result of Rasmussen \cite{RasDiff} the reduced and unreduced triply graded homology determine each other.
Given a collection of vector spaces indexed by three gradings, one gets a complicated three-dimensional array of their dimensions. We solve this problem by introducing $\Delta$-grading (similar to $\delta$-grading in Khovanov homology), which is a linear combination of the three gradings, and presenting the homology in each $\Delta$-grading separately. This breaks the three-dimensional array into two-dimensional slices, see Figure \ref{11n_80}. In this and other figures throughout the paper we show $a$- and $q$-grading (which correspond to the variables in the HOMFLY-PT polynomial) respectively vertically and horizontally, and mark $\Delta$-grading for each slice - this determines the $t$-grading. The numbers indicate the ranks of the corresponding homology groups\footnote{Note that this is very different from \cite{DGR} and some other sources where the numbers indicate the $t$-grading.}.

\begin{figure}[ht!]
    \begin{center}
        \includegraphics[width=\textwidth]{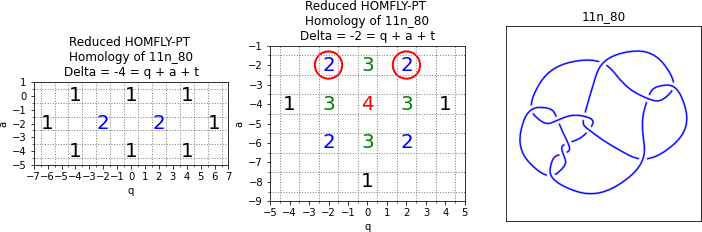}
    \end{center}
    \caption{Triply graded homology of the knot $K=11n_{80}$. The $q$ and $a$ gradings are plotted, and $t$ is determined by either $q+a+t=-4$ or $q+a+t=-2$. 
    The red circles indicate the homology of the differentials $d_1$ and $d_{-1}$, in particular, the HOMFLY-PT $S$-invariant equals 2.
    }
    \label{11n_80}
\end{figure}

Following Rasmussen \cite{RasTwoBridge}, we call a knot {\em thin}, if $\HHH(K)$ is supported in a single $\Delta$-grading. 
By the results of \cite{RasTwoBridge} all two-bridge knots are thin and   $\HHH(K)$ is determined by the HOMFLY-PT polynomial and the signature. 
It turns out that all knots in the dataset \cite{NS} are almost thin, that is, supported in at most two $\Delta$-gradings. 
More precisely, the distribution of $\Delta$-gradings is shown in Table \ref{tab: delta distribution}.
\begin{table}
\begin{center}
    \begin{tabular}{|l|c|c|}
    \hline
               & $|\Delta|=1$ & $|\Delta|=2$ \\
               \hline
    Two-bridge     &  173 & 0\\
    \hline
    Alternating, not two-bridge & 293 & 1\\
    \hline
    Not alternating & 137& 91\\
    \hline
    Total & 603 & 92\\
    \hline
    \end{tabular}
\end{center}
\caption{The only alternating knot in the dataset with $|\Delta|=2$ is $11a_{263}$, see Figure \ref{fig: 11a263}.}
\label{tab: delta distribution}
\end{table}


In this note, we argue that even if $K$ is not thin, in many cases (and for most knots in the dataset) it is possible to recover a lot of useful information about $\sll(N)$ homology and various related invariants by simply looking at $\HHH(K)$. 

Our first main result describes $S$-invariants which can be defined using a triply graded analogue of the Lee spectral sequence in Khovanov homology. 

\begin{theorem}
\label{thm: intro S}
For all knots in the dataset, the $S$-invariant is determined by $\HHH(K)$.
\end{theorem}

For example, for the knot $11n_{80}$ in Figure \ref{11n_80} the $S$-invariant equals 2, and corresponds to the (right) red circle in $\Delta$-grading $(-2)$ and $(q,a)=(2,-2)$.

We prove Theorem \ref{thm: intro S} in Section \ref{sec: data} as Corollary \ref{cor: S determined}, the main observation is that using all three gradings allows us to conclude that there is only one nontrivial differential in the spectral sequence, and we can pin down the location of its homology. This contrasts the situation in Khovanov homology where computing $s$-invariant is highly nontrivial.  

Our next results concerns the relation between  $\HHH(K)$ and (reduced) $\sll(N)$ Khovanov-Rozansky homology $H_{\sll(N)}(K)$ for arbitrary knots. By \cite{RasDiff} there is a spectral sequence from $\HHH(K)$ to $H_{\sll(N)}(K)$. Furthermore, it follows from \cite{LL,RasDiff} that for every monic polynomial $\partial W$ of degree $N$ there exists a spectral sequence with $E_2$ page isomorphic to $H_{\sll(N)}(K)$ and one-dimensional $E_\infty$ page. In \cite{LL}
Lewark and Lobb used this construction to define a family of link invariants $s_{\partial W}$ which, in principle, depend on the choice of polynomial (see Section \ref{sec: small example}). For $N=2$ there is only one such invariant $s_2=s_{x^2-x}$ and with our conventions the Rasmussen $s$-invariant equals $s=2s_2$.
 

\begin{theorem}
\label{thm: intro sN stabilize}
For $N\ge 3$ the $\sll(N)$ homology of all knots in the dataset is isomorphic to $\HHH(K)$ up to regrading. For $\partial W=x^N-x$, the corresponding invariant $s_{x^N-x}(K)$ is determined by the $S$-invariant, which is known by Theorem \ref{thm: intro S}.
\end{theorem}

For $N=2$ or for more complicated knots, the behaviour of Rasmussen spectral sequence could be more subtle, see Section \ref{sec: example}. Nevertheless, for all but two knots in the dataset the Rasmussen $s$-invariant in Khovanov homology   agrees with the HOMFLY-PT $S$-invariant up to a sign. For knots $K$ which are supported in a single $\Delta$-grading, we know that $S(K) = -s(K)$ and that $\dim(\HHH(K)) = \dim(H_{\sll(2)}(K))$.
In Table \ref{tab: S and s} we show the distributions of values of $S(K)+s(K)$ and $\dim(\HHH(K)) - \dim(H_{\sll(2)}(K))$ among the 92 knots in the dataset which are not supported in a single $\Delta$-grading. 

\begin{table}[ht!]

\begin{tabular}{|l|c|c|c|c|c|c|c|c|c|c|c|c|}
\hline
$\dim\HHH(K) - \dim\HHH_{\sll(2)}(K)$ & 0  & 2 & 6 & 8  & 10 & 14 & 16 & 22 & 24 & 26 & 32 & 34 \\ \hline
Number of Knots                    & 14 & 2 & 4 & 27 & 4  & 1  & 19 & 7  & 10 & 1  & 2  & 1  \\ \hline 
\end{tabular}

\vspace{4mm}

\begin{tabular}{|l|c|c|c|}
\hline
$S(K)+s(K)$       & -2 & 0  & 2 \\ \hline
Number of Knots & 1  & 90 & 1 \\ \hline
\end{tabular}

\vspace{4mm}

\caption{The two knots whose $S$-invariant value differs from the $s$-invariant are $10_{125}$ with $S(10_{125})+s(10_{125}) = -2$ and $11n_{82}$ with $S(11n_{82})+s(11n_{82}) = 2$.}
\label{tab: S and s}
\end{table}

The following result is more abstract, and uses the action of the Lie algebra $\sll(2)$ in triply graded homology constructed in \cite{GHM}. We denote its generators by $(E,H,F)$ with the convention that $E$ increases the $q$-grading by $4$ and $F$ decreases it by $4$. The action of $\sll(2)$ preserves the $\Delta$-grading and the $a$-grading, so it acts horizontally in each $\Delta$-grading in all figures. One immediate consequence of the existence of such action, proved in \cite{GHM}, is {\em symmetry} of $\HHH(K)$ and {\em unimodality}  in each remainder of $q$-grading modulo 4.
This is clearly visible in Figure \ref{11n_80}.

In this note, we record another consequence of the $\sll(2)$ action. Let $d_N$ be the first differential in the spectral sequence from $\HHH(K)$ to $\sll(N)$ homology.

\begin{theorem}
\label{thm: intro diff relations}
a) There exists a family of operators $d_{a|b}$ on $\HHH(K)$ satisfying the following equations:
$$
d_{N|0}=d_N,\ [E,d_{a|b}]=b\cdot d_{a+1|b-1},\ [F,d_{a|b}]=a\cdot d_{a-1|b+1},\ [H,d_{a|b}]=(a-b)\cdot d_{a|b}.
$$
The symmetry of $\HHH(K)$ exchanges $d_{a|b}$ with $d_{b|a}$.

b) For $N=1$ the operators $d_1=d_{1|0}$ and $d_{-1}=d_{0|1}$ satisfy additional relations
$$
d_{-1}^2=d_{-1}d_1+d_1d_{-1}=0.
$$
\end{theorem}

We illustrate the action of $d_{a|b}$ in the homology of the 15-crossing torus knot $T(4,5)$ in Section \ref{sec: T45}.

\begin{corollary}
\label{cor: intro E diffs}
The differentials $d_N$ commute with the action of $E$.
\end{corollary}

In fact, we first prove Corollary \ref{cor: intro E diffs} as Theorem \ref{thm: E diffs}, and then deduce Theorem \ref{thm: intro diff relations} from it in Lemma \ref{lem: super diffs}.

By Theorem \ref{thm: intro diff relations}(b), the differentials $d_{1}$ and $d_{-1}$ anticommute and span a two-dimensional representation of $\sll(2)$. By Theorem \ref{thm: intro S} the action of $d_1$ (and hence of $d_{-1}$) is completely determined by $\HHH(K)$ for all knots in the dataset. In Proposition \ref{prop: blocks} we use these ideas to decompose $\HHH(K)$ in symmetric blocks and explain some patterns in triply graded homology.

For all knots in the dataset, the homology of $d_1$ is one-dimensional (by Proposition \ref{prop: d1 exceptions}) and supported in $\Delta$-grading $(-S)$ and bidegree $(q,a)=(S,-S)$. Similarly, the homology of $d_{-1}$ is one-dimensional and supported in the same $\Delta$-grading and $(q,a)=(-S,-S)$. We illustrate both of these homologies by red circles in all figures and write the value of the $S$-invariant for the reader's convenience.

\section*{Acknowledgments}

We thank Lukas Lewark, Allison Moore, Jake Rasmussen, Paul Wedrich and Melissa Zhang for helpful discussions and suggestions. The reduced Khovanov homology and $\sll(3)$ homology were computed using Lukas Lewark's programs \cite{FoamHo,KhoCa}, we thank Lukas for help with running them. E. G. was partially supported by the NSF grant DMS-1760329, and by the NSF grant DMS-1928930 while E. G. was in residence at the Simons Laufer Mathematical Sciences Institute (previously known as MSRI) in Berkeley, California, during the Fall 2022 semester.

\section{Link homology data}
\label{sec: data}

In this section we discuss in detail the link homology data provided by \cite{NS}, and its interaction with $\sll(N)$ Khovanov-Rozansky homology and the associated spectral sequences defined by Rasmussen in \cite{RasDiff}. 

We postpone the formal definitions of HOMFLY-PT homology and Rasmussen spectral sequences to Section \ref{sec: theory}, where we prove some new general properties of differentials in these spectral sequences. 
In this section, we mostly use these as a black box, recording only the grading shifts and show that even this limited information in many cases yields very precise answers. 

\subsection{Gradings on $\HHH$}

In this section we collect and compare several grading conventions from various sources. 
We use the same skein relation for the (normalized) HOMFLY-PT polynomial $P(L)$ for links $L$ as Nakagane and Sano \cite{NS}:

$$a^{-1}P(\usebox\poscrossing)-aP(\usebox\negcrossing)=(q^{-1}-q)P(\usebox\onesmoothing) \ \ \ \ \ \ \ \ P(\usebox\unknot)=1.$$

This differs from the HOMFLY-PT polynomial $P_\DGR$ used in \cite{DGR} by $P_\DGR(q,a)=P(q,a^{-1})$, so $P(L) = P_\DGR(\bar{L})$ where $\bar{L}$ denotes the mirror of the link $L$.

Unless stated otherwise, we will always work in   {\bf reduced} triply graded homology with rational coefficients. By the universal coefficient theorem, this is equivalent to working with any other characteristic zero field (such as $\mathbb{C}$) and we will not make this distinction unless stated otherwise.
The triply graded homology $\HHH$ has gradings $(q,a,t)$. Our grading choice follows Nakagane and Sano \cite{NS} which is a bit unusual compared with other sources. 

The Poincar\'e polynomial for $\HHH(L)$ is defined as $\CP(L)(q,a,t)=\sum_{i,j,k}q^ia^jt^{\frac{k-j}{2}}\dim\HHH^{i,j,k}(L)$
and we have $P(L)(q,a)=\CP(L)(q,a,-1)$ for all links $L$ \cite[Section 1]{NS}.

\begin{remark}\label{note_on_gradings}
One slightly confusing note: we often go back and forth between thinking about the Poincar\'e polynomial and the homology, which are two different ways of presenting the same data. When referring to homology, the letters $q,a,t$ refer to the values of the gradings (e.g. if $q=1, a=2, t=3$ we are talking about $\HHH^{1,2,3}(L)$) but when talking about Poincar\'e polynomials, these same values become the powers on the formal variables $q,a,t$ (e.g. if $q=1, a=2, t=3$ we are talking about the term $qa^2t^3$). 
\end{remark}

Another choice of gradings is discussed in the paper of Dunfield, Gukov and Rasmussen \cite{DGR}. The $q$ and $a$ gradings are the same, but the third grading is $t_{\DGR}=\frac{1}{2}(a-t)$.
For knots, the gradings $q,a,t$ are all even and therefore $t_{\DGR}$ is an integer. 


For signatures, we use the same convention as in \cite{DGR} that a negative knot has positive signature, so in particular, if $K$ is the negative trefoil, $\sigma(K)=2$.

Finally, for any $N>0$ we will use $\sll(N)$-modified $q$-grading $q_{\sll(N)}=q+Na$. As we will see below in Proposition \ref{prop: differentials}, this agrees with the $q$-grading on $\sll(N)$ Khovanov-Rozansky homology up to a certain spectral sequence.

\subsection{Delta grading and thin knots}
 
The delta grading $\Delta=q+a+t$ will be extremely useful for our work. It is different from the $\delta$-grading used in \cite{DGR} by a factor of $(-2)$:
$$
\delta_{\DGR}=-\frac{\Delta}{2}=t_{\DGR}-\frac{q}{2}-a.
$$

A reader might be more familiar with the $\delta$-grading in $\sll(2)$ homology defined by 
$$
\delta_{\sll(2)}=\frac{q_{\sll(2)}}{2}-t_{\DGR}=\frac{q}{2}+a-t_{\DGR}.
$$
This agrees with the grading $\frac{\Delta}{2}$, up to a spectral sequence from  Proposition \ref{prop: differentials}.

We call a knot $K$ {\bf thin} (or KR-thin as in \cite{RasTwoBridge}) if $\HHH(K)$ is supported in a single $\Delta$-grading. We will denote by $\HHH_{\Delta}(K)$ the subspace of $\HHH(K)$ in delta grading $\Delta$.
We recall an important result of Rasmussen:

\begin{theorem}[\cite{RasTwoBridge}]
Any two-bridge knot is thin and supported in the $\Delta$-grading equal to negative its signature.
\end{theorem}

 \begin{figure}
    \begin{center}
        \includegraphics[width=0.8\textwidth]{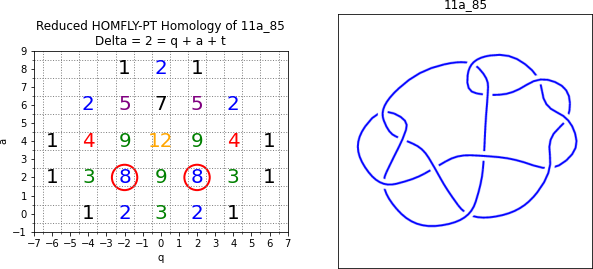}
    \end{center}
    \caption{Triply graded homology of the knot $K=11a_{85}$. Since $K$ is two-bridge, it follows that $K$ is thin. The $q$ and $a$ gradings are plotted, and $t$ is determined by $q+a+t=2$. 
    The $S$-invariant equals $-2$.
    }
    \label{two-bridge-11a85}
\end{figure}

Figure \ref{two-bridge-11a85} shows an example of HOMFLY-PT homology for a two-bridge link. We define the $\Delta$-thickness $|\Delta|(K)$ as the number of consecutive $\Delta$-gradings supporting $\HHH(K)$, that is,
$$
|\Delta|(K)=\frac{1}{2}\left(\max\{\Delta:\HHH_{\Delta}\neq 0\} -\min\{\Delta:\HHH_{\Delta}\neq 0\}\right)+1.
$$
For thin knots we have $|\Delta|=1$. Recall all two-bridge knots are alternating.
Figure \ref{two-bridge-11a85} shows an example of HOMFLY-PT homology for a two-bridge knot $11a_{85}$ supported in a single delta grading $\Delta=2$.

 We show the distribution of $\Delta$-gradings for knots in the dataset in Table \ref{tab: delta distribution}. The only alternating knot supported in two $\Delta$-gradings in the dataset is $11a_{263}$, see Figure \ref{fig: 11a263}.

\begin{remark}
It is easy to see that there exist knots with arbitrarily large $\Delta$-thickness. For example, it follows from \cite{Mellit} that the triply graded homology for $(p,p+1)$ torus knot contains classes with both $(q,a,t_{\DGR})=(-p(p-1),p(p-1),0)$ and $(q,a,t_{\DGR})=(0,p(p+1)-2,p^2-1)$ with $\Delta=p(p-1)$ and $\Delta=2p-2$ respectively, so $|\Delta|(T(p,p+1))\ge \frac{(p-1)(p-2)}{2}+1$.
One could also use connect sums of knots to get arbitrary large $|\Delta|$.
\end{remark}

 \begin{figure}
    \begin{center}
        \includegraphics[width=\textwidth]{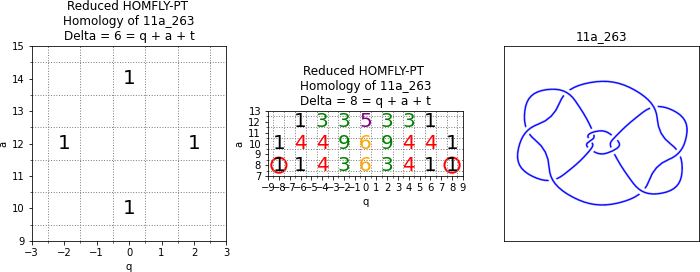}
    \end{center}
    \caption{Triply graded homology of the knot $K=11a_{263}$. This is the only alternating knot in the dataset which is not thin. The $S$-invariant equals $-8$.
    }
    \label{fig: 11a263}
\end{figure}

We call a knot $K$ {\bf parity} if $t/2$ has the same parity on $\HHH(K)$. Equivalently, $t_{\DGR}-\frac{a}{2}$ has the same parity and 
$a+q-\Delta$ has the same remainder modulo 4.
The main result of \cite{Mellit} implies that all torus knots are parity.

\subsection{Differentials and $\sll(N)$ homology}

We will use differentials $d_N$ on HOMFLY-PT homology and the associated spectral sequences, see the details below in Section \ref{sec: theory}, in particular Theorem \ref{thm: Rasmussen ss}. We will denote the reduced $\sll(N)$ homology by $H_{\sll(N)}(K)$.

\begin{proposition}[\cite{RasDiff}]
\label{prop: differentials}
There exists a spectral sequence with $E_1$ page isomorphic to  $\HHH(K)$ and   $i$-th differential $d_N^{(i)}$ changing the degrees by $(q,a,t)\to (q+2Ni,a-2i,t+2-2i)$. The $E_{\infty}$ page is isomorphic to $H_{\sll(N)}(K)$ with modified $q$-grading $q_{\sll(N)}=q+Na$.
\end{proposition}

\begin{remark}
The $t$-grading on $\sll(N)$ homology which we use here agrees with $t_{\DGR}$.   
\end{remark}

We will denote the first differential in this spectral sequence by $d_N=d_N^{(1)}$. In our conventions it changes degrees by $(q,a,t)\to(q+2N,a-2,t)$. The differentials $d_N,d_{N'}$ anticommute for different $N,N'$.

\begin{corollary}
\label{cor: dN delta}
The differential $d_N^{(i)}$ increases the $\Delta$-grading by $2Ni-4i+2$ and decreases $t_{\DGR}$ by 1.
\end{corollary}

\begin{corollary}
\label{cor: dN parity}
If $K$ is parity then the differentials $d_N^{(i)}$ vanish for even $i$.
\end{corollary}





As a special case of Proposition \ref{prop: differentials}, we can consider $N=1$. The $\sll(1)$ homology of any knot is one-dimensional and supported in bidegree $(q_{\sll(1)},t_{\DGR})=(0,0)$. Since $q_{\sll(1)}=q+a$, we get the following:

\begin{corollary}
\label{cor: def S}
a) There is a spectral sequence starting at $\HHH(K)$ and converging to $E_{\infty}\simeq H_{\sll(1)}(K)\simeq \Q$. 
The surviving homology group is supported in tridegree $(q,a,t)=(S,-S,-S)$ and $\Delta$-grading equal to $-S$. 

b) The differentials $d_1^{(i)}$ change the degrees by 
$(q,a,t)\to (q+2i,a-2i,t+2-2i)$ and increase the $\Delta$-grading by $2-2i$.
\end{corollary}

The $S$-invariant of $K$ \cite{DGR} is defined as the $q$-grading of surviving generator of $H_{\sll(1)}(K)$.
One can define the analogues of $S$-invariant for $\sll(N)$ homology using the following theorem. 

\begin{theorem}[\cite{Lobb1,Lobb2,LL,Wu}]
\label{thm: slN Lee}
Let $\partial W$ be a degree $N$ monic polynomial with simple root at 0. Then there is a spectral sequence starting at $H_{\sll(N)}(K)$ and converging to the one-dimensional homology $H_{\partial W}$, which depends on $\partial W$.
\end{theorem}

See Theorem \ref{thm: W s invariant} for more details and Section \ref{sec: small example} for some examples for various $\partial W$.

\begin{definition}
Let $j_{\partial W}$ be the $q_{\sll(N)}$-degree of the surviving generator in the spectral sequence from Theorem \ref{thm: slN Lee}. Define
$$
s_{\partial W}=\frac{j_{\partial W}}{2(N-1)}.
$$
\end{definition}

In particular, for $N=2$ the invariant $s_{\partial W}$ does not depend on the choice of $\partial W$  (see Remark \ref{rem: s2 unique}) and we recover the celebrated $s$-invariant \cite{RasSlice} by $s=2s_2$.
The importance of the invariants $s_{\partial W}$ is shown by the following result: 

\begin{theorem}[\cite{Lobb1,Lobb2,LL,Wu}]
The slice genus $g_*(K)$ satisfies the inequality 
$$
g_*(K)\ge |s_{\partial W}(K)|
$$
\end{theorem}

As a consequence, an efficient way of computing $s_{\partial W}$ 
for various $N$ and various polynomials $\partial W$ would give a collection of slice genus bounds for the knot $K$.

\begin{remark}
In \cite{LL} Lewark and Lobb proposed more subtle slice genus bounds from {\bf unreduced} $\sll(N)$ homology and its deformations. It is unlikely if these can be computed by methods of this paper since unreduced HOMFLY-PT homology is infinite-dimensional and supported in infinitely many $\Delta$-gradings.
\end{remark}

\begin{proposition}
\label{prop: N large}
Assume that $\HHH(K)$ has $\Delta$-thickness $|\Delta|$. Then 
$$
H_{\sll(N)}(K)\simeq \HHH(K),\ s_{x^N-x}(K)=-\frac{1}{2}S(K)
$$
for all $N>|\Delta|$.
\end{proposition}

\begin{proof} Since $|\Delta|\ge 1$ we get $N\ge 2$.
By Corollary \ref{cor: dN delta} the differential $d_N^{(i)}$ changes the $\Delta$-grading by $2Ni-4i+2=2(N-2)i+2$. For $N>|\Delta|$ and $i\ge 1$ we have $2(N-2)i+2\ge 2(N-2)+2=2(N-1)\ge 2|\Delta|$, but the difference between the $\Delta$-gradings of any two classes is at most $2|\Delta|-2$. So all differentials $d_N^{(i)}$ must vanish and the spectral sequence collapses.

To compare the $S$-invariants, observe that the spectral sequence for $\partial W=x^N-x$ is induced by the differential $d_1$ on $H_{\sll(N)}$, see Corollary \ref{cor: dN d1} below (since $d_N$ and $d_1$ anticommute, the action of $d_1$ is well defined on $H_{\sll(N)}$). Therefore the spectral sequences of Corollary \ref{cor: def S} and Theorem \ref{thm: slN Lee} agree. The surviving generator for the former is supported in degrees $(q,a,t)=(S,-S,-S)$ and therefore has $q_{\sll(N)}$ degree $j_{x^N-x}=S+N(-S)=-(N-1)S$.
Now 
$$
s_{x^N-x}=\frac{j_{x^N-x}}{2(N-1)}=-\frac{1}{2}S.
$$
\end{proof}

\begin{corollary}
\label{cor: diffs collapse}
Let $K$ be a knot.
\begin{enumerate}[label=\alph*)]
\item If $|\Delta|=1$ then $H_{\sll(N)}(K)\simeq \HHH(K)$ for all $N\ge 2$.
\item If $|\Delta|=2$ then $H_{\sll(N)}(K)\simeq \HHH(K)$ for all $N\ge 3$. 
\end{enumerate}
In particular, for all knots in the dataset \cite{NS} we have $H_{\sll(3)}(K)\simeq \HHH(K)$.
\end{corollary}

\begin{example}
The HOMFLY-PT and $\sll(3)$ homology of the knot $9_{11}$ are shown in Figure \ref{homfly-vs-sl3}. Since $|\Delta|=1$, the spectral sequence to $\sll(N)$ homology collapses for all $N\geq 2$, and $\HHH(K)\simeq H_{\sll(N)}(K)$ as vector spaces. However, the regrading changes the picture of the HOMFLY-PT homology (Figure \ref{homfly-vs-sl3}, left) to the $\sll(3)$ homology (Figure \ref{homfly-vs-sl3}, right) significantly.

 \begin{figure}[ht!]
    \begin{center}
        \includegraphics[width=\textwidth]{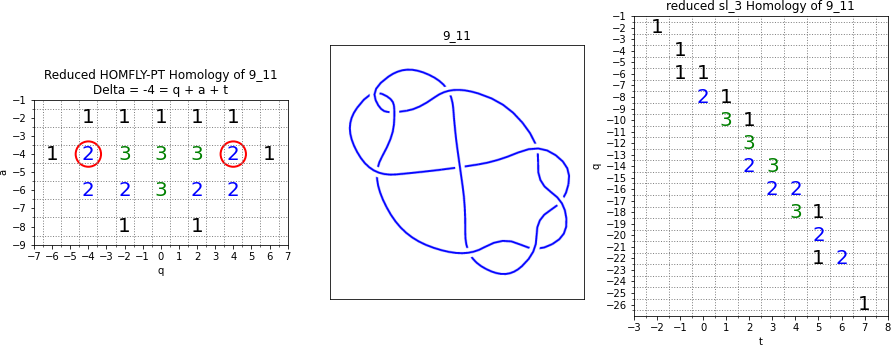}
    \end{center}
    \caption{Here we have both the triply graded homology and $\sll(3)$ homology of the knot $K=9_{11}$. The $S$-invariant equals $4$ while the $\sll(3)$ invariants are equal to $j_{x^3-x}=-8$, $s_{x^3-x}=\frac{j_{x^3-x}}{4}=-2$.
    }
    \label{homfly-vs-sl3}
\end{figure}

\end{example}

Next, we address the case $N=1$. We call a knot {\em $d_1$-standard} if all higher differentials $d_1^{(i)}$ in the $\sll(1)$ spectral sequence
vanish for $i\ge 2$.

\begin{lemma}
\label{lem: d1 standard}
Let $K$ be a knot.
\begin{enumerate}[label=\alph*)]
\item Assume that $\HHH(K)$ is supported in a single $\Delta$-grading. Then $K$ is $d_1$-standard.

\item Assume that $\HHH(K)$ is supported in two neighboring $\Delta$-gradings $(\Delta,\Delta+2)$. Furthermore, assume that 
$$
\max\{a: \HHH^{a,\Delta+2}\neq 0\}-4 < \min\{a: \HHH^{a,\Delta}\neq 0\}.
$$
Then $K$ is $d_1$-standard. 
\end{enumerate}
\end{lemma}

\begin{proof}
The  differential $d_1^{(i)}$ changes the $\Delta$-grading by $2-2i$, so (a) is clear and in (b) the only possible higher differential is $d_1^{(2)}$.
It shoots from $(a,\Delta+2)$ to $(a-4,\Delta)$, and by our assumptions  there is no room for it. 
\end{proof}

\begin{proposition}
\label{prop: d1 exceptions}
All knots in the dataset \cite{NS} are $d_1$-standard. 
\end{proposition}


\begin{proof}
We checked using computer that all knots in \cite{NS} satisfy the assumptions of Lemma \ref{lem: d1 standard}(b) with the following 17 exceptions:
$$
10_{128},\
10_{136},\
11n_{104},\
11n_{12},\
11n_{126},\
11n_{133},\
11n_{145},\
11n_{155},\
11n_{16},\
11n_{20},\
$$
$$
11n_{39},\
11n_{45},\
11n_{57},\
11n_{61},\
11n_{64},\
11n_{79},\
11n_{9}.
$$
Recall that $d_1^{(2)}$ changes the degrees by $(a,q,\Delta)\to (a-4,q+4,\Delta-2)$, so by looking at $q$- and $a$-degrees we can exclude $10_{136},11n_{12},11n_{20},11n_{79}$ (see Section \ref{sec: appendix} for details).

We show the triply graded homology for all remaining  exceptions in  Section \ref{sec: appendix}. By direct inspection, we see that if a potential differential $d_1^{(2)}$ is nonzero then it has rank 1 and acts from the grading $(q_K,a_K,\Delta+2)$ to $(q_K+4,a_K-4,\Delta)$, where $q_K,a_K$ are determined by the knot and listed in the following table:

\begin{center}
\begin{tabular}{|c|c|c|c|c|c|c|c|c|}
\hline
$K$ & $10_{128}$ & $11n_{9}$ & $11n_{16}$ & $11n_{39}$ &  $11n_{45}$ & $11n_{57}$ & $11n_{61}$ & $11n_{64}$ \\
\hline 
$(q_K,a_K)$ & (-4,10) & (-4,8) & (-4,8) & (-4,2) & (-4,2) & (-4,8) & (-4,6) & (-4,6) \\
\hline
$K$ & $11n_{104}$ & $11n_{126}$ &
$11n_{133}$ & $11n_{145}$ & $11n_{155}$  & & &   \\
\hline
$(q_K,a_K)$ & (-4,8) & (-4,10) & (-4,6) & (-4,2) & (-4,4)   & & &\\
\hline
\end{tabular}
\end{center}

Consider the direct sum of triply graded homology with $q_{\sll(1)}=q+a=q_K+a_K$, it breaks into two pieces: $A$ in delta grading $\Delta$ and $B$ in delta grading $\Delta+2$. The differential $d_1=d_1^{(1)}$ preserves both $A$ and $B$ while $d_1^{(2)}$ potentially acts from $B$ to $A$. 

In all exceptional cases but $11n_{155}$, the dimensions of $A$ and $B$ are even (and, in fact, $\chi(A)=\chi(B)=0$), so 
the dimensions of both $H(A,d_1)$ and $H(B,d_1)$ are even as well. If $d_1^{(2)}$ has rank 1 then it must have odd-dimensional kernel and cokernel, and the dimension of the $E_{\infty}$ page with $q_{\sll(1)}=q_K+a_K$ is at least 2. Contradiction.

Finally, for $11n_{155}$ the ranks of graded components of  $A$ and $B$ have the following form:
\begin{center}
\begin{tikzpicture}[scale=0.5]
\draw (0,2) node {$1$};
\draw (2,0) node {$1$};
\draw (6,4) node {$1$};
\draw (8,2) node {$5$};
\draw (10,0) node {$5$};
\draw (12,-2) node {$2$};
\draw (-0.5,2)--(0,2.5)--(2.5,0)--(2,-0.5)--(-0.5,2);
\draw (5.5,4)--(6,4.5)--(12.5,-2)--(12,-2.5)--(5.5,4);
\draw (1,-1) node {$A (\Delta=0)$};
\draw (11,-3) node {$B (\Delta=2)$};
\draw [->,dashed] (5.7,3.7)--(2.3,0.3);
\end{tikzpicture}
\end{center}
The differential $d_1$ acts in southeast direction for both $A$ and $B$, and the potential higher differential $d_1^{(2)}$ is shown by the  dashed arrow. If $d_1^{(2)}\neq 0$ then it is easy to see that the rank of homology at the $E_{\infty}$ page is at least $3$, contradiction. 

\end{proof}

\begin{lemma}
\label{lem: S determined}
Assume that $K$ is $d_1$-standard. Then the HOMFLY-PT $S$-invariant and the ranks of $d_1$ in each trigrading are completely determined by $\HHH(K)$.
\end{lemma}

\begin{proof}
By Corollary \ref{cor: def S}
 the $\sll(1)$ spectral sequence converges to the $E_{\infty}$ page which is 1-dimensional and supported in   $\Delta$-grading $-S$.

Since $K$ is $d_1$-standard and the differential $d_1$ preserves $\Delta$-grading, for $\Delta\neq -S$ it is acyclic and for $\Delta=-S$ it has 1-dimensional homology. By computing the Euler characteristic for each $\Delta$-grading, we compute $S$. This determines the position of 1-dimensional homology of $d_1$, and $d_1$ is acyclic on the complement, so its ranks are completely determined by $\HHH(K)$.
\end{proof}

\begin{corollary}\label{cor: S determined}
For all knots $K$ in the dataset \cite{NS}:
\begin{enumerate}[label=\alph*)]
\item the HOMFLY-PT $S$-invariant and the ranks of $d_1$ at every trigrading are completely determined by $\HHH(K)$.

\item the  invariant $s_{x^N-x}$ is completely determined by $\HHH(K)$ for $N\ge 3$.
\end{enumerate}
\end{corollary}

\begin{proof}
Part (a) is immediate from Proposition \ref{prop: d1 exceptions} and Lemma \ref{lem: S determined}. Part (b) follows from Proposition \ref{prop: N large}.
\end{proof}

\begin{example}
As a warning to the reader, the HOMFLY-PT $S$-invariant is not necessarily equal to the signature if $K$ is not thin. For example, for the knot $9_{42}$ the $S$-invariant equals 0 while the signature equals 2.
\end{example}

 For $N=2$ the situation is more complicated, see Table \ref{tab: S and s}. Still, for all but two knots in the dataset the HOMFLY-PT and $\sll(2)$ invariants agree up to sign. We show some examples of computations of $d_2$ in Section \ref{sec: examples}.

\begin{figure}
    \begin{center}
        \includegraphics[width=\textwidth]{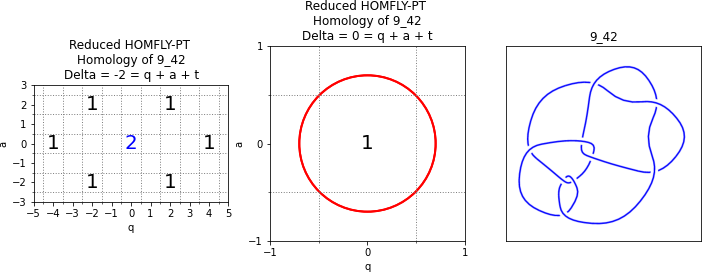}
    \end{center}
    \caption{Triply graded homology of the knot $K=\textnormal{$9_{42}$}$. This is an example of a knot whose $S$ invariant, $0$, is not equal to the negative of its signature, $2$. 
    }
    \label{fig: 9 42}
\end{figure}

\section{Differentials, symmetry and $\sll(2)$ action}
\label{sec: theory}

In \cite{GHM} the first author, Hogancamp and Mellit defined an action of the Lie algebra $\sll(2)$ on HOMLFY homology. We review its definition in Section \ref{sec: sl2 definition}, and prove some of its new properties in Section \ref{sec: sl2 properties}. These allow us to understand better the structure of HOMFLY-PT homology.

First, we review in detail the construction of HOMFLY-PT homology and its $y$-ification, and the construction of spectral sequences to $\sll(N)$ homology both for HOMFLY-PT and $y$-ified homology.

\subsection{Soergel bimodules and Rouquier complexes}

Let $R=\C[x_1,\ldots,x_n]$ be the polynomial ring in $n$ variables. It is graded by $\deg(x_i)=q^2$. To any $n$-strand braid we will associate a complex of $R$-$R$ bimodules, where the two actions of $R$ correspond, respectively, to the marked points on the bottom and top of the braid.  Consider a bimodule
$$
B_i:=R\bigotimes_{R^{(i\ i+1)}}R=\frac{\C[x_1,\ldots,x_n,x'_1,\ldots,x'_n]}{x_i+x_{i+1}=x'_i+x'_{i+1},x_ix_{i+1}=x'_ix'_{i+1},x_j=x'_j\quad (j\neq i,i+1)}
$$
The {\em Rouquier complexes} associated to single crossings are given by 
$$
T_i=[B_i\xrightarrow{b_i} R],\quad T_i^{-1}=[qR\xrightarrow{b_i^{*}} B_i]
$$
where $b_i:B_i\to R$ and $b_i^{*}:qR\to B_i$ are morphisms of $R$-$R$ bimodules which send $1\mapsto 1$ and $1\mapsto x_i-x'_{i+1}$, respectively. Note that the bimodules $B_i$ are graded, and the differentials in $T_i^{\pm}$ preserve the grading. 

\begin{theorem}[\cite{Rouquier}]
The complexes $T_i,T_i^{-1}$ satisfy braid relations up to homotopy:
$$
T_i\otimes T_i^{-1}\simeq R,\quad T_i\otimes T_{i+1}\otimes T_{i}\simeq T_{i+1}\otimes T_i\otimes T_{i+1},\quad T_i\otimes T_j\simeq T_j\otimes T_i\quad (|i-j|\ge 2).
$$
The tensor product is considered over $R$.
\end{theorem}

As a consequence, one can associate a complex (known as Rouquier complex) of $R$-$R$ bimodules to an arbitrary braid $\beta=\sigma_{i_1}^{\epsilon_1}\cdots \sigma_{i_r}^{\epsilon_r}$ by   $T_{\beta}=T_{i_1}^{\epsilon_1}\cdots T_{i_r}^{\epsilon_r}$.
The following construction of ``dot-sliding homotopies" is well known, but we spell it out in detail it for the reader's convenience.

\begin{proposition}
\label{prop: def xi}
The left and right actions of $R$ on $T_{i}^{\pm}$ are homotopic up to the transposition $s_i=(i\ i+1)$. 
More specifically, for all $a$ there is a map $\xi_a:T_i^{\pm}\to T_i^{\pm}$ of homological degree $(-1)$ and $q$-degree 2, such that
$$
[d,\xi_a]=x_a-x'_{s(a)}
$$
\end{proposition}

\begin{proof}
For $T_i$ we define $\xi_i=b_i^{*}, \xi_{i+1}=-b_i^{*}$ and $\xi_j=0$ for $j\neq i,i+1$. For $T_i^{-1}$ we define $\xi_i=b_i, \xi_{i+1}=-b_i$ and $\xi_j=0$ for $j\neq i,i+1$. All the properties follow from the identity $b_ib_i^{*}=b_i^{*}b_i=x_i-x'_{i+1}$.
\end{proof}

\begin{corollary}
For any braid $\beta$ the left and right action of $R$ on the complex $T_{\beta}$ are homotopic, up to the permutation corresponding to $\beta$.
\end{corollary}

Given an $R-R$ bimodule $M$, we can consider its {\bf Hochschild      
 homology} $\HH(M)$, defined as follows: consider an arbitrary free $R\otimes R$-resolution of $M$:
$$
\left[\ldots \rightarrow M_2\rightarrow M_1 \rightarrow M_0\right]\simeq M
$$
identify the left and right $R$-action on each $M_i$ by writing
$$
\left[\ldots \rightarrow M_2\bigotimes_{R\otimes R}R\rightarrow M_1\bigotimes_{R\otimes R}R \rightarrow M_0\otimes_{R\otimes R}R\right]
$$
and then take homology of the resulting complex of $R$-modules.

\begin{definition}
\label{def: HHH}
Let $\beta$ be an arbitrary braid.
The  triply graded homology $\HHH(\beta)$ is defined applying Hochschild homology to the Rouquier complex $T_{\beta}$, and computing the homology of the result:
$$
\overline{\HHH}(\beta):=H_*(\HH(T_{\beta}))
$$
where the functor $\HH(-)$ is applied to the bimodules in the complex $T_{\beta}$ term-wise. 
\end{definition}

\begin{remark}
To be precise, this defines $\HHH(\beta)$ up to an overall grading shift. In this section we will not need precise formulas for this shift, and refer to \cite{NS} for the precise shift in their grading conventions. 
\end{remark}

By \cite[Section 2.8]{RasDiff} 
the homology $\overline{\HHH}(\beta)$ is a free module over $\C[x_1+\ldots+x_n]$ and we can define {\bf reduced triply graded homology} as the quotient 
$$ 
\HHH(\beta)=\overline{\HHH}(\beta)/(x_1+\ldots+x_n)\overline{\HHH}(\beta).
$$ 
If $\beta$ closes to a knot then $\HHH(\beta)$ is finite dimensional, and this is indeed the link homology considered above.

In what follows we will need a variant of this construction due to Rasmussen \cite{RasDiff}. Given a braid $\beta=\sigma_{i_1}^{\epsilon_1}\cdots \sigma_{i_\ell}^{\epsilon_{\ell}}$, the associated Rouquier complex 
$T_{\beta}=T_{i_1}^{\epsilon_1}\cdots T_{i_\ell}^{\epsilon_{\ell}}$ is a tensor product of $\ell$ $R-R$ bimodules. As such, it carries an action of $\ell+1$ copies of $R$, which motivates the following definition.

\begin{definition}
Let $D=\widehat{\beta}$ be a knot diagram obtained by closing a braid $\beta$. We associate a variable $x_{i}^{(j)}$ to any edge in $D$, where $1\le i\le n, 0\le j\le \ell$ there are $(\ell+1)n$ variables in total. The {\bf reduced edge ring} $R_e$ is obtained as the quotient of $R[x_i^{(j)}]$ by the following three sets of linear relations:
\begin{itemize}
    \item[(a)] At each crossing $\sigma_{i_k}^{\epsilon_k}$ we have 
    $$
    x_{i_k}^{(k-1)}+x_{i_k+1}^{(k-1)}=x_{i_k}^{(k)}+x_{i_k+1}^{(k)},\quad x_j^{(k-1)}=x_j^{(k)}\quad (j\neq i_k,i_{k+1}).
    $$
    \item[(b)] The variables at the top and bottom are identified:
    $$
    x_i^{(0)}=x_i^{(\ell)}
    $$
    \item[(c)] The sum of variables at the bottom vanishes: $\sum_{i}x_i^{(0)}=0$.
    \end{itemize}
One can also consider the {\bf unreduced edge ring} using equations (a) and (b), but not (c).

\end{definition}
\begin{example}
Consider the braid $\beta=\sigma_1\sigma_2^{-1}\sigma_1\sigma_2^{-1}$ which closes to the figure-eight knot:

\begin{center}
\begin{tikzpicture}
\draw (0,1)..controls (0,0.5) and (1,0.5)..(1,0);
\draw [line width=5,white] (0,0)..controls (0,0.5) and (1,0.5)..(1,1);
\draw (0,0)..controls (0,0.5) and (1,0.5)..(1,1);
\draw (2,0)--(2,1);

\draw (1,1)..controls (1,1.5) and (2,1.5)..(2,2);
\draw [line width=5,white] (1,2)..controls (1,1.5) and (2,1.5)..(2,1);
\draw (1,2)..controls (1,1.5) and (2,1.5)..(2,1);
\draw (0,1)--(0,2);

\draw (0,3)..controls (0,2.5) and (1,2.5)..(1,2);
\draw [line width=5,white] (0,2)..controls (0,2.5) and (1,2.5)..(1,3);
\draw (0,2)..controls (0,2.5) and (1,2.5)..(1,3);
\draw (2,2)--(2,3);

\draw (1,3)..controls (1,3.5) and (2,3.5)..(2,4);
\draw [line width=5,white] (1,4)..controls (1,3.5) and (2,3.5)..(2,3);
\draw (1,4)..controls (1,3.5) and (2,3.5)..(2,3);
\draw (0,3)--(0,4);

\draw (2,4)..controls (2,4.3) and (3,4.3)..(3,4);
\draw (1,4)..controls (1,4.7) and (4,4.7)..(4,4);
\draw (0,4)..controls (0,5) and (5,5)..(5,4);

\draw (2,0)..controls (2,-0.3) and (3,-0.3)..(3,0);
\draw (1,0)..controls (1,-0.7) and (4,-0.7)..(4,0);
\draw (0,0)..controls (0,-1) and (5,-1)..(5,0);

\draw (3,0)--(3,4);
\draw (4,0)--(4,4);
\draw (5,0)--(5,4);

\draw (-0.1,0)--(0.1,0);
\draw (0.9,0)--(1.1,0);
\draw (1.9,0)--(2.1,0);
\draw (-0.1,1)--(0.1,1);
\draw (0.9,1)--(1.1,1);
\draw (1.9,1)--(2.1,1);
\draw (-0.1,2)--(0.1,2);
\draw (0.9,2)--(1.1,2);
\draw (1.9,2)--(2.1,2);
\draw (-0.1,3)--(0.1,3);
\draw (0.9,3)--(1.1,3);
\draw (1.9,3)--(2.1,3);
\draw (-0.1,4)--(0.1,4);
\draw (0.9,4)--(1.1,4);
\draw (1.9,4)--(2.1,4);

\draw (-0.4,0) node {$x_1^{(0)}$};
\draw (0.6,0) node {$x_2^{(0)}$};
\draw (1.6,0) node {$x_3^{(0)}$};
\draw (-0.4,1) node {$x_1^{(1)}$};
\draw (0.6,1) node {$x_2^{(1)}$};
\draw (1.6,1) node {$x_3^{(1)}$};
\draw (-0.4,2) node {$x_1^{(2)}$};
\draw (0.6,2) node {$x_2^{(2)}$};
\draw (1.6,2) node {$x_3^{(2)}$};\
\draw (-0.4,3) node {$x_1^{(3)}$};
\draw (0.6,3) node {$x_2^{(3)}$};
\draw (1.6,3) node {$x_3^{(3)}$};
\draw (-0.4,4) node {$x_1^{(4)}$};
\draw (0.6,4) node {$x_2^{(4)}$};
\draw (1.6,4) node {$x_3^{(4)}$};
\end{tikzpicture}
\end{center}

The relations in the edge ring are
$$
x_1^{(0)}+x_2^{(0)}=x_1^{(1)}+x_2^{(1)},\ x_2^{(1)}+x_3^{(1)}=x_2^{(2)}+x_3^{(2)},\
x_1^{(2)}+x_2^{(2)}=x_1^{(3)}+x_2^{(3)},\ x_2^{(3)}+x_3^{(3)}=x_2^{(4)}+x_3^{(4)},\
$$
$$
x_3^{(0)}=x_3^{(1)},x_1^{(1)}=x_1^{(2)},
x_3^{(2)}=x_3^{(3)},x_1^{(3)}=x_1^{(4)},
x_1^{(0)}=x_1^{(4)},x_2^{(0)}=x_2^{(4)},x_3^{(0)}=x_3^{(4)},
x_1^{(0)}+x_2^{(0)}+x_3^{(0)}=0.
$$
\end{example}


Note that equations (a) and (c) imply that $\sum_{i}x_i^{(j)}=0$ for all $j$.
We also define the {\bf local edge ring} as 
$$
R_{loc,i}:=\frac{\C[x_1,\ldots,x_n,x'_1,\ldots,x'_n]}{x_i+x_{i+1}=x'_i+x'_{i+1},\quad x_j=x'_j\quad (j\neq i,i+1)}.
$$
It is clear that the relations (a) define a tensor product of $R_{loc,i_{k}}$ over  appropriate polynomial rings, while $R_e$ is obtained as a quotient of this product by the additional relations (b) and (c).

\begin{lemma}
a) The $R$-$R$ bimodules $R$ and $B_i$ have the following resolutions over $R_{loc,i}$:
\begin{equation}
\label{eq: res R}
R\simeq \left[R_{loc,i}\xrightarrow{x_i-x'_i} R_{loc,i}\right],
\end{equation}
\begin{equation}
\label{eq: res B}
B_i\simeq \left[R_{loc,i}\xrightarrow{(x'_i-x_i)(x'_i-x_{i+1})} R_{loc,i}\right].
\end{equation}
b) The morphisms $b_i:B_i\to R$ and $b_i^*:R\to B_i$ lift to the following morphisms of resolutions:
\begin{equation}
\label{eq: Rouquier resolved}
b_i\simeq
\begin{tikzcd}
 R_{loc,i}\arrow{rr}{x_i-x'_i} & & R_{loc,i}\\
 R_{loc,i}\arrow{rr}{(x'_i-x_i)(x'_i-x_{i+1})} \arrow{u}{x_{i+1}-x'_i}& & R_{loc,i} \arrow{u}{1}\\
\end{tikzcd}
\quad 
\quad
b_i^*\simeq 
\begin{tikzcd}
 R_{loc,i}\arrow{rr}{(x'_i-x_i)(x'_i-x_{i+1})} & & R_{loc,i}\\
 R_{loc,i}\arrow{rr}{x_i-x'_i} \arrow{u}{1}& & R_{loc,i} \arrow{u}{x_{i+1}-x'_i}\\
\end{tikzcd}
\end{equation}
\end{lemma}

\begin{proof}
For \eqref{eq: res R}, note that $x_i-x'_i=0$ is equivalent to $x_{i+1}-x'_{i+1}=0$ in $R_{loc,i}$. For \eqref{eq: res B}, note that for any symmetric function $\varphi$ in two variables we have $\varphi(x_i,x_{i+1})|_{B_i}=\varphi(x'_i,x'_{i+1})|_{B_i}$. In particular, 
$(x'_i-x_i)(x'_i-x_{i+1})=(x'_i-x'_i)(x'_{i}-x'_{i+1})=0$ in $B_i$, so we can write
$$
B_{i}=\frac{\C[x_1,\ldots,x_n,x'_1,\ldots,x'_n]}{x_i+x_{i+1}-x'_{i}-x'_{i+1}=0,\ (x'_i-x_i)(x'_i-x_{i+1})=0, \quad x_j=x'_j\quad (j\neq i,i+1)}=$$ $$\frac{R_{loc,i}}{(x'_i-x_i)(x'_i-x_{i+1})=0}
$$
and the result follows. 

For part (b) note that the diagrams commute, so we get chain maps of resolutions which induce the desired maps on homology.
\end{proof}

\begin{lemma}
\label{lem: Rasmussen HHH}
a) The tensor product of resolutions \eqref{eq: res B} for $B_{i_1},\ldots, B_{i_\ell}$
yields a resolution of $B_{i_1}\otimes \cdots B_{i_\ell}$ over the tensor product of $R_{loc,i_k}$.

b) The reduced HOMFLY-PT homology $\HHH(\beta)$ can be computed as follows. First, tensor the cones of $b_i$ and $b_i^{*}$ given by \eqref{eq: Rouquier resolved}. Second, impose the relations (b),(c) from definition of $R_e$, this  gives a bicomplex of $R_e$-modules. Take homology of the  horizontal differential, then homology of the vertical differential.
\end{lemma}

\begin{proof}
a) It is well known that the bimodule $B_{i_1}\otimes \cdots B_{i_\ell}$ is free both as a left and right $R$-module (this follows from the fact that $R$ is free over $R^{(i\ i+1)}$). Therefore the derived tensor product of such bimodules is quasi-isomorphic to the usual tensor product.

b) By (a), to compute the Hochschild homology $\HH(B_{i_1}\otimes \cdots B_{i_\ell})$ it is sufficient to replace each $B_i$ by its resolution \eqref{eq: res B}, tensor these resolutions and identify the variables on the top and bottom. The rest follows from the definition of $\HHH(\beta)$.
\end{proof}

\subsection{Rasmussen spectral sequences}

Following \cite{RasDiff}, we deform the above construction of $\HHH$ to define a spectral sequence to $\sll(N)$ homology. Let us fix an integer $N$. We will need a more general version which depends on a one-variable polynomial (called the {\bf potential}) $W(x)$. The $\sll(N)$ homology corresponds to $W(x)=x^{N+1}$.

For the crossing $\sigma_i^{\epsilon}$ we define
$$
W_{i}:=W(x_i)+W(x_{i+1})-W(x'_i)-W(x'_{i+1})=\sum_{j=1}^{n}\left(W(x_j)-W(x'_j)\right)\in R_{loc,i}
$$
Observe that $W_i$ vanishes both when $x_i=x'_i$ (and $x_{i+1}=x'_{i+1}$ by definition of $R_{loc,i}$), and when $x_{i+1}=x'_i$ (and $x_{i}=x'_{i+1}$ by definition of $R_{loc,i}$).
Since $R_{loc,i}$ is a free polynomial ring and $(x_i-x'_i),(x_{i+1}-x'_i)$ are coprime in it, we conclude that $W_i$ is divisible by $(x_i-x'_i)(x_{i+1}-x'_i)$. We denote
$$
W'_{i}=\frac{W_{i}}{x_i-x'_i},\ W''_{i}=\frac{W_{i}}{(x_i-x'_i)(x_{i+1}-x'_i)}.
$$

One can replace the resolutions \eqref{eq: res R} and \eqref{eq: res B} by the following:
$$
\begin{tikzcd}
R_{loc,i}\arrow[bend left]{rr}{x_i-x'_i} & &  R_{loc,i} \arrow[bend left]{ll}{W'_{i}} &
R_{loc,i}\arrow[bend left]{rr}{(x'_i-x_i)(x'_i-x_{i+1})} & & R_{loc,i}\arrow[bend left]{ll}{W''_{i}}.
\end{tikzcd}
$$
Note that these are not complexes but {\bf matrix factorizations}. More precisely, if $d_{+}$ denotes the rightward differential (which agrees with the differentials in \eqref{eq: res R} and \eqref{eq: res B} respectively), and $d_{\partial W}$ denotes the leftward differential then
$$
(d_{+}+d_{\partial W})^2=W_{i}.
$$
The maps $b_i$ and $b_i^*$ from \eqref{eq: Rouquier resolved} are unchanged. One can then proceed with defining the complex as in Lemma \ref{lem: Rasmussen HHH}, replacing the horizontal differential with $d_{+}+d_{\partial W}$. Note that on the tensor product we have 
$$
(d_{+}+d_{\partial W})^2=\sum_{j=1}^{n}\left(W\left(x_j^{(0)}\right)-W\left(x^{(\ell)}_j\right)\right)=0\ \mathrm{in}\ R_e.
$$

Following \cite{LL}, we  label  the differential $d_{\partial W}$ and the corresponding homology $H_{\partial W}$ by the derivative $\partial W=\frac{dW}{dx}$ rather than by $W$. This is explained by the following:

\begin{example}
For a given potential $W$, the matrix factorization for a trivial braid on one strand is given by  
$$
\begin{tikzcd}
\C[x,x'] \arrow[bend left]{rr}{x-x'} & & \C[x,x'] \arrow[bend left]{ll}{\frac{W(x)-W(x')}{x-x'}}
\end{tikzcd}
$$
Since $\frac{W(x)-W(x')}{x-x'}|_{x=x'}=\partial W$, after closing the braid we get the complex:
$$
\begin{tikzcd}
\C[x] \arrow[bend left]{rr}{0} & & \C[x] \arrow[bend left]{ll}{\partial W}
\end{tikzcd}
$$

and the unreduced homology of unknot:
$$
H_{\partial W}(\mathrm{unknot})=\C[x]/\left(\partial W\right).
$$
\end{example}

\begin{theorem}[\cite{RasDiff}]
\label{thm: Rasmussen ss}
a) For any $W$, the homology $H_{\partial W}:=H(H(\CS(D),d_{+}+d_{\partial W}),d_v)$ is a link invariant. The differential $d_{\partial W}$ induces a spectral sequence from $\HHH(\beta)$ to $H_{\partial W}(\beta)$.

b) For $\partial W=x^{N}$, the homology $H_{\partial W}$ agrees with (reduced) $\sll(N)$ Khovanov-Rozansky homology $H_{\sll(N)}(\beta)$. The differential $d_{\partial W}$ and the corresponding spectral sequence agree with the differential $d_N$ and the spectral sequence from Proposition \ref{prop: differentials}.
\end{theorem}

\begin{remark}
\label{rem: Rasmussen ss details}
In what follows we will need some details of the proof of Theorem \ref{thm: Rasmussen ss} concerning the order of differentials defining spectral sequence which we recall now. First, we define 
$$
H^{+}(\CS)=H(\CS(D),d_{+}),\quad H^{\pm}(\CS)=H(H^+(\CS),d_{\partial W})=H(H(\CS(D),d_{+}),d_{\partial W}).
$$
By \cite[Corollary 5.9]{RasDiff} $H^{\pm}(\CS)$ is supported in a single horizontal grading. By \cite[Lemma 5.12]{RasDiff} this implies that the spectral sequence 
\begin{equation}
\label{eq: ss plus minus}
H^{\pm}(\CS)=H(H(\CS(D),d_{+}),d_{\partial W})\Rightarrow H(\CS(D),d_{+}+d_{\partial W})
\end{equation}
collapses and induces a canonical isomorphism between these homology groups. Furthermore, the spectral sequence 
\begin{equation}
\label{eq: ss minus v}
H(H^{\pm}(\CS),d_{v})=H(H(H^+(\CS),d_{\partial W}),d_{v})\Rightarrow H(H^+(\CS),d_{\partial W}+d_{v})
\end{equation}
collapses as well \cite[Proposition 5.10]{RasDiff} since the differential $d_{\partial W}$ changes the horizontal grading by $(-1)$, the differential $d_v$ preserves the horizontal grading and therefore higher differentials must increase the horizontal grading which is impossible by the above. 

To sum up, the collapse of spectral sequences \eqref{eq: ss plus minus} and \eqref{eq: ss minus v} implies a chain of isomorphisms
\begin{equation}
\label{eq: Rasmussen iso}
H_{\partial W}(\beta)\simeq H(H(\CS(D),d_{+}+d_{\partial W}),d_v)\simeq H(H^{\pm}(\CS),d_{v})\simeq H(H^+(\CS),d_{\partial W}+d_{v}).
\end{equation}
On the other hand, by Lemma \ref{lem: Rasmussen HHH} we have a spectral sequence
$$
\HHH(\beta)=H(H^+(\CS),d_{v})\Rightarrow H(H^+(\CS),d_{\partial W}+d_{v})\simeq H_{\partial W}(\beta).
$$
induced by $d_{\partial W}$.
\end{remark}

As above, for $\partial W=x^{N}$ we denote the differentials in the spectral sequence by $d_N^{(i)}$ and write $d_N=d_N^{(1)}$. By \cite[Corollary 5.6]{RasDiff}
the differentials $d_M,d_N$ anticommute for different $M,N$.

\begin{lemma}
\label{lem: xi commutes with slN}
The dot-sliding homotopies $\xi_i$ from Proposition \ref{prop: def xi} commute with both differentials $d_{+}$ and $d_{\partial W}$ above.
\end{lemma}

\begin{proof}
We can lift $\xi_i$ to the resolutions \eqref{eq: Rouquier resolved} as follows:
$$
\label{eq: Rouquier resolved xi}
\begin{tikzcd}
 R_{loc,i}\arrow[bend left]{rr}{x_i-x'_i} \arrow{ddd}[bend left]{1}& & R_{loc,i} \arrow[bend left]{ll}{W'_{i}} \arrow[bend left]{ddd}{x_{i+1}-x'_i}\\
  & & \\
    & & \\
 R_{loc,i}\arrow[bend left]{rr}{(x'_i-x_i)(x'_i-x_{i+1})} \arrow[bend left]{uuu}{x_{i+1}-x'_i}& & R_{loc,i} \arrow[bend left]{ll}{W''_{i,N}} \arrow{uuu}{1}\\
\end{tikzcd}
\quad 
\quad
\begin{tikzcd}
 R_{loc,i}\arrow[bend left]{rr}{(x'_i-x_i)(x'_i-x_{i+1})} \arrow{ddd}[bend left]{x_{i+1}-x'_i}& & R_{loc,i} \arrow[bend left]{ll}{W''_{i}} \arrow[bend left]{ddd}{1}\\
  & & \\
    & & \\
 R_{loc,i}\arrow[bend left]{rr}{x_i-x'_i} \arrow[bend left]{uuu}{1}& & R_{loc,i} \arrow[bend left]{ll}{W'_{i}} \arrow{uuu}{x_{i+1}-x'_i}\\
\end{tikzcd}
$$
Here $\xi_i$ are denoted by downward arrows, and the vertical differential $d_v$ by upward arrows, as above.
It is easy to see that 
$$
[\xi_i,d_{+}]=[\xi_i,d_{\partial W}]=0, [d_v,\xi_i]=x_{i+1}-x'_i,
$$
and the result follows. 
\end{proof}

\subsection{Generalizations of the $s$-invariant}

In \cite{LL} Lewark and Lobb defined a family of link invariants $s_{\partial W,\alpha}$ depending on a choice of a potential $W$ and a complex number $\alpha$ satisfying $\partial W(\alpha)=0$ (so that $\alpha$ is a critical point of $W$).  By \cite[Proposition 3.3]{LL} the invariants for $W(x)$ and $W(x+c)$ are equal, so throughout the paper we will assume $\alpha=0$. Furthermore, we can assume that $\partial W$ is monic and $W(0)=0$,  and write
\begin{equation}
\label{eq: W}
W(x)=\frac{x^{N+1}}{N+1}+a_N\frac{x^N}{N}+\ldots+a_2\frac{x^2}{2},\quad 
\partial W=x^N+a_Nx^{N-1}+\ldots+a_2x.
\end{equation}

The following is a variation of \cite[Theorem 2.4]{LL}, adapted to our setting of HOMFLY-PT homology.

\begin{theorem}
\label{thm: W s invariant}
a) Assume that $a_2\neq 0$ (so that $\alpha=0$ is a simple root of $\partial W$). Then the reduced homology $H_{\partial W}(K)$ of any knot $K$ is one-dimensional. 

b) There is a spectral sequence from $\HHH(K)$ to $H_{\partial W}(K)$ with the first differential $$d_{\partial W}=d_{N}+a_{N}d_{N-1}+\ldots+a_2d_1$$

c) There is a spectral sequence from $H_{\sll(N)}(K)$ to $H_{\partial W}(K)$ with the first differential $$a_{N}d_{N-1}+\ldots+a_2d_1$$
\end{theorem}

\begin{proof}
Part (b) is immediate from Theorem \ref{thm: Rasmussen ss}, since $d_{\partial W}$ is linear in $W$.

To prove parts (a) and (c), we define two different filtrations on the corresponding chain complex. Observe that the differential $d_{m}$ changes $(q,a)$ bidegree to $(q+2m,a-2)$, so it changes $q_{\sll(k)}$ by $2m-2k=2(m-k).$

To prove (a), we use $q_{\sll(1)}$ to define a filtration. Indeed,
the term $a_2d_1$ preserves $q_{\sll(1)}$ and other terms $d_N+\ldots+a_3d_2$ strictly increase $q_{\sll(1)}$.  Since the homology of $d_1$ is one-dimensional, the spectral sequence of the filtered complex collapses, and $H_{\partial W}(K)$ is one-dimensional as well.

To prove (c), we use $q_{\sll(N)}$ to define a filtration. Indeed,
the term $d_N$ preserves $q_{\sll(N)}$ and other terms $a_Nd_{N-1}+\ldots+a_2d_1$ strictly decrease $q_{\sll(N)}$, so we get a desired spectral sequence for the filtered complex.
\end{proof}

Following \cite{LL}, we denote by $j_{\partial W}(K)$ the $q_{\sll(N)}$-degree of the surviving generator of the $E_{\infty}$ page in the spectral sequence in Theorem \ref{thm: W s invariant}(c)  and write
$$
s_{\partial W}=\frac{j_{\partial W}}{2(N-1)}.
$$ 

\begin{remark}
\label{rem: def j}
One can make the definition of $j_{\partial W}$ more explicit as follows. By the proof of Theorem \ref{thm: W s invariant}(c), the underlying chain complex $C$ is filtered by $q_{\sll(N)}$,  and one can consider subcomplexes
$$
\mathcal{F}_jC=\{u\in C: q_{\sll(N)}(u)\le j\}.
$$
This filtration induces a filtration on the homology $H_{\partial W}$: $\mathcal{F}_jH_{\partial W}$ consists of homology classes which have representatives in $\mathcal{F}_jC$. The $E_{\infty}$ page of the spectral sequence is then given by
$$
E_{\infty}=\bigoplus_{j}\mathcal{F}_jH_{\partial W}/\mathcal{F}_{j-1}H_{\partial W}.
$$
Since $H_{\partial W}$ is one-dimensional, $j_{\partial W}$ can be simply defined by
$$
j_{\partial W}=\min\{j: \mathcal{F}_jH_{\partial W}\neq 0\}.
$$
\end{remark}

\begin{remark}
\label{rem: s2 unique}
For $N=2$ we have $\partial W=x^2+a_2x$ with $a_2\neq 0$. By \cite[Proposition 3.3]{LL} one can assume $a_2=-1$, so that $\partial W=x^2-x$. Therefore for $N=2$ all invariants $s_{\partial W}$ agree and we write $s_{x^2-x}=s_2$. 
\end{remark}

\begin{corollary}
\label{cor: dN d1}
For $\partial W=x^{N}-x$ there is a spectral sequence from $H_{\sll(N)}$ to $H_{x^N-x}$ induced by the differential $d_1$.
\end{corollary}

Another interesting choice of potential corresponds to $\partial W=x^N-1$ and $\alpha=1$. We can replace it by $\partial W=(x+1)^N-1$ and $\alpha=0$. 

\begin{example}
For $N=3$ we get $(x+1)^3-1=x^3+3x^2+3x$, so there is a spectral sequence from $H_{\sll(3)}$ to $H_{(x+1)^3-1}$ induced by the differential $3(d_2+d_1)$.
\end{example}

As shown in \cite{LL}, the invariants $s_{x^N-x}$ and $s_{(x+1)^N-1}$ could in fact differ, see Section \ref{sec: small example} for an example. 

\subsection{$y$-ification}

Following \cite{GH}, we define the $y$-ification of  triply graded homology. We introduce additional formal variables $y_i$ (associated to strands of the braid) and tensor all chain groups by $\C[y_i]$. In terms of Rouquier complexes, we deform the differential as 
$$
D=d+\sum y_i\xi_i,\ D^2=\sum y_i(x_i-x'_{w(i)}).
$$
After closing the braid and identifying $y_i$ on the same connected component of the link, we get a well defined complex since $D^2=0$ \cite[Lemma 3.3]{GH}. The rest of Definition \ref{def: HHH} goes through, and the definition of Hochschild homology is unchanged.
We will denote the $y$-ified complex with differential $D$ by $\CY(\beta)$ and its  homology (also known as {\bf $y$-ified homology} by $\HY(\beta)$. One of the main results of \cite{GH} proves that $\HY(\beta)$ is a topological invariant of the closure of $\beta$.

In the notations of Lemma \ref{lem: Rasmussen HHH}, we deform the vertical differential by $D_v=d_v+\sum y_i\xi_i$ and do not change the horizontal differential $d_{+}$. By replacing the horizontal differential by $d_{+}+d_{\partial W}$, we get a definition of  $y$-ified  homology $\HY_{\partial W}$. For $\partial W=x^N$, this yields $y$-ified  $\sll(N)$ homology (see also \cite{CK,BS}).  

\begin{theorem}
\label{thm: Rasmussen ss deformed}
The above construction yields a well-defined $y$-ified  homology  $\HY_{\partial W}$ where the horizontal differential $d_{+}+d_{\partial W}$ commutes with the vertical differential $D_v$. There is a spectral sequence from $\HY(\beta)$ to  $\HY_{\partial W}(\beta)$ similar to the one in Theorem \ref{thm: Rasmussen ss}.
\end{theorem}

\begin{proof}
We follow the logic of the proof of Theorem \ref{thm: Rasmussen ss} outlined in Remark \ref{rem: Rasmussen ss details}. By  Lemma \ref{lem: xi commutes with slN} the differential $D_v$ anticommutes with both $d_+$ and $d_{\partial W}$. The definitions of $H^+$ and $H^{\pm}$ are unchanged, so $H^{\pm}$ is supported in a single horizontal degree. The spectral sequence \eqref{eq: ss plus minus} is unchanged (and collapses), while the spectral sequence \eqref{eq: ss minus v} is replaced by
$$ 
H(H^{\pm}(\CS),D_{v})=H(H(H^+(\CS),d_{\partial W}),D_{v})\Rightarrow H(H^+(\CS),d_{\partial W}+D_{v})
$$ 
The differential $D_v$ preserves the horizontal degree, so
by the same argument as in Remark \ref{rem: Rasmussen ss details} this spectral sequence collapses. This implies a chain of isomorphisms
\begin{equation}
\label{eq: Rasmussen iso deformed}
H_{\sll(N)}(\beta)\simeq H(H(\CS(D),d_{+}+d_{\partial W}),D_v)\simeq H(H^{\pm}(\CS),D_{v})\simeq H(H^+(\CS),d_{\partial W}+D_{v}).
\end{equation}
and a spectral sequence
$$
\HY(\beta)=H(H^+(\CS),D_{v})\Rightarrow H(H^+(\CS),d_{\partial W}+D_{v})\simeq \HY_{\partial W}(\beta).
$$
induced by $d_{\partial W}$.
\end{proof}

\begin{corollary}
\label{cor: operator}
Suppose that $\Phi: \CS(\beta)\otimes \C[y_i]\to \CS(\beta)\otimes \C[y_i]$ commutes with $D_v,d_+$ and $d_{\partial W}$. Then $\Phi$ defines an operator on $\HY(\beta)$ which commutes with the differentials in the spectral sequence from Theorem \ref{thm: Rasmussen ss deformed}.
\end{corollary}

\begin{proof}
Since $\Phi$ commutes with $d_+$ and $d_{\partial W}$, it induces well-defined operators on $H^+(\CS)\otimes \C[y_i]$ and $H^{\pm}(\CS)\otimes \C[y_i]$. Since $\Phi$ commutes with $D_v$ as well,
it commutes with all the isomorphisms in \eqref{eq: Rasmussen iso deformed}, and the result follows.
\end{proof}

\subsection{Action of $\sll(2)$: definition}
\label{sec: sl2 definition}

\begin{lemma}
On any Rouquier complex $T_{\beta}$ there exists an operator $u$ of homological degree $(-2)$ and $q$-degree 4 such that 
\begin{equation}
\label{eq: def u}
[d,u]=\sum_{a=1}^{n} (x_a+x'_{w(a)})\xi_a
\end{equation}
where $w$ is the permutation corresponding to $\beta$.  
\end{lemma}

\begin{proof}
We define $u=0$ for the braid generators $T_i^{\pm}$. This satisfies \eqref{eq: def u} for $T_i$ since 
$$
\sum_{a=1}^{n} (x_a+x'_{w(a)})\xi_a=b_i^{*}(x_i+x'_{i+1}-x_{i+1}-x'_{i})=0,
$$
the check for $T_i^{-1}$ is similar. Given two braids $\beta$ and $\gamma$ with the corresponding permutations $v,w$, homotopies $\xi_a^{\beta}$ and $\xi_a^{\gamma}$ and operators $u_{\beta},u_{\gamma}$, we define
\begin{equation}
\label{eq: gluing u}
u_{\beta\gamma}=u_{\beta}+u_{\gamma}+\sum_{a=1}^{n}\xi_{a}^{\beta}\xi_{v(a)}^{\gamma}.
\end{equation}
We have
$$
[d,u_{\beta\gamma}]=\sum_{a=1}^{n} (x_a+x'_{v(a)})\xi^{\beta}_a+
\sum_{a=1}^{n} (x'_{v(a)}+x''_{wv(a)})\xi^{\gamma}_{v(a)}+\sum_{a=1}^{n}(x_a-x'_{v(a)})\xi_{v(a)}^{\gamma}-\sum_{a=1}^{n}(x'_{v(a)}-x''_{wv(a)})\xi_{v(a)}^{\beta}=
$$
$$
\sum_{a=1}^{n} (x_a+x''_{wv(a)})\xi^{\beta}_a+\sum_{a=1}^{n} (x_a+x''_{wv(a)})\xi^{\gamma}_{v(a)}=
\sum_{a=1}^{n} (x_a+x''_{wv(a)})(\xi^{\beta}_a+\xi^{\gamma}_{v(a)}).
$$
On the other hand, $\xi_{a}^{\beta\gamma}=\xi_a^{\beta}+\xi_{v(a)}^{\gamma}$.
This allows us define $u$ inductively for arbitrary products of $T_i^{\pm}$.
\end{proof}

\begin{remark}
If we unpack the above proof and use \eqref{eq: gluing u} repeatedly, we will get the following elementary formula:
\begin{equation}
\label{eq: u easy}
u=\sum_{c\prec c'}\alpha_{c,c'}\xi^{(c)}\xi^{(c')}
\end{equation}
where $c,c'$ are two crossings in a braid $\beta$, $\xi^{(c)}$ and $\xi^{(c')}$ are single dot-sliding homotopies at $c,c'$ (given by $b_i$ or $b^*_i$) and 
$\alpha_{c,c'}$ is the sum of the following terms:
$$
\begin{cases}
+1 & \text{if the right output strand of}\ c\  \text{connects to the left input strand of}\ c'\\
-1 &  \text{if the right output strand of}\ c\  \text{connects to the right input strand of}\ c'\\
-1 &  \text{if the left output strand of}\ c\  \text{connects to the left input strand of}\ c'\\
+1 &  \text{if the left output strand of}\ c\  \text{connects to the right input strand of}\ c'\\
\end{cases}
$$
\end{remark}

\begin{proof}
It follows from \eqref{eq: gluing u} that
$$
u=\sum_{c\prec c'}\sum_{a}\xi_a^{(c)}\xi_{w(a)}^{(c')}
$$
where $w$ is the permutation matching the input strands of $c$ with the input strands of $c'$. It remains to notice that the $\xi_a^{(c)}$ for the left input (and the right output) strand of $c$ corresponds to $\xi^{(c)}$, the right input (and the left output) strand of $c$ corresponds to $-\xi^{(c)}$, and all other strands contribute 0. 
\end{proof}

Now we can define the operator on $y$-ified homology:
$$
E=\sum_a \left(x_a+x'_{w(a)}\right)\frac{\partial}{\partial y_i}+u
$$
\begin{theorem}[\cite{GHM}]
\label{thm: sl2}
We get the following:

a) $[D,E]=0$, so $E$ defines a chain map on the $y$-ified complex $\CY(\beta)$

b) $E$ induces an endomorphism of the $y$-ified homology $\HY(\beta)$ which is an invariant of the link obtained by closure of $\beta$. 

c) There exists an endomorphism $F$ of $\HY(\beta)$ such that $(E,F,H)$ form an $\sll(2)$-triple, and $H=\frac{1}{2}\deg_{q}$.

d) If $\beta$ closes up to a knot $K$, the Lie algebra $\sll(2)$ generated by $E,F,H$ also acts in the reduced HOMFLY-PT homology of $K$.
\end{theorem}

\begin{remark}
In \cite{GHM} the operator $E$ was denoted by $F_2$. Here we chose to change notations to match representation theory of $\sll(2)$ better. In particular, we want to emphasize that $E$ increases the $q$-grading, so one should think of it as a raising operator. 
\end{remark}

\begin{remark}
As explained above, one can define $E$  on chain level by a fairly explicit formula. On the contrary, the construction of $F$ is not explicit, and follows from the ``hard Lefshetz property" for $E$.
\end{remark}

\subsection{Action of $\sll(2)$: properties}
\label{sec: sl2 properties}

The first and the most important consequence from the existence of $\sll(2)$ action is that triply graded homology is symmetric, as is any finite-dimensional $\sll(2)$-representation.

\begin{theorem}[\cite{GHM}]
For any knot $K$, the triply graded homology $\HHH(K)$ is symmetric around the vertical axis. The symmetry preserves the $\Delta$-grading and transforms the three gradings by $(q,a,t)\to (-q,a,t+2q)$.
\end{theorem}

\begin{theorem}
\label{thm: E diffs}
a) The operator $E$ from $\sll(2)$ commutes with all differentials in the $y$-ified $\sll(N)$ spectral sequence from Theorem \ref{thm: Rasmussen ss deformed}.

b) Suppose that $\beta$ closes to a knot. Then the operator $E$ from $\sll(2)$ commutes with all differentials in the $\sll(N)$ spectral sequence (in particular, with $d_N$).
\end{theorem}

\begin{proof}
By  Theorem \ref{thm: sl2} $E$ commutes with the differential $D_v$. By \eqref{eq: u easy} we can write the operator $u$ as a linear combination of products of $\xi_i$, hence by Lemma \ref{lem: xi commutes with slN} we have $[u,d_+]=[u,d_{\partial W}]=0$. Therefore $[E,d_{+}]=[E,d_{\partial W}]=0$ and by Corollary \ref{cor: operator} this implies part (a). Part (b) follows from (a) and Theorem \ref{thm: sl2}. 
\end{proof}

\begin{lemma}
\label{lem: super diffs}
There exists a family of operators $d_{a|b}$ on $\HHH(K)$ satisfying the following equations:
$$
d_{N|0}=d_N,\ [E,d_{a|b}]=b\cdot d_{a+1|b-1},\ [F,d_{a|b}]=a\cdot d_{a-1|b+1},\ [H,d_{a|b}]=(a-b)\cdot d_{a|b}.
$$
The symmetry of $\HHH(K)$ exchanges $d_{a|b}$ with $d_{b|a}$.
\end{lemma}

\begin{proof}
If $d_{N}=0$ then we can set $d_{a|b}=0$ for all $a+b=N$.
Otherwise, assume $d_N\neq 0$ on $\HHH(K)$.  

The endomorphism algebra $\End(\HHH(K))$ is a finite-dimensional representation of $\sll(2)$ which acts by $\ad_E,\ad_F,\ad_H$. The eigenvalue of $\ad_H$ on an endomorphism $A$ equals half of $q$-degree of $A$, in particular, $\ad_H(d_N)=Nd_N$. Furthermore, $\ad_E(d_N)=0$ by Theorem \ref{thm: E diffs}(b), so $d_N$ is a highest weight vector of weight $N$, and hence spans an $(N+1)$-dimensional representation of $\sll(2)$.

More precisely, for $a+b=N$ we can define 
$$
d_{a|b}:=\frac{1}{b!}\ad_F^{b}(d_N)\in\End(\HHH(K)),
$$
then the $\sll(2)$ relations are satisfied and 
$$
[F,d_{0|N}]=\frac{1}{N!}\ad_F^{N+1}(d_N)=0.
$$
\end{proof}

See Section \ref{sec: T45} for an example of computation of $d_{a|b}$.

\begin{remark}
One might expect that the operators $d_{a|b}$ agree with the conjectural ``supergroup differentials'' conjectured in \cite{GGS,GNR}, but we do not pursue it here. 
\end{remark}

\begin{corollary}
Define $d_{-1}=\ad_F(d_1)=[F,d_1]$. Then 
\begin{equation}
\label{eq: d-1 sl2}
[E,d_{-1}]=d_1, [F,d_{-1}]=0,
\end{equation}
so that $d_1,d_{-1}$ span a 2-dimensional representation of $\sll(2)$. Furthermore,
\begin{equation}
\label{eq: d-1 anticommute}
d_{-1}^2=d_{-1}d_1+d_1d_{-1}=0.
\end{equation}
The symmetry of $\HHH(K)$ exchanges $d_1$ with $d_{-1}$, in particular, $d_{-1}$ is a cancelling differential if $d_1$ is. 
\end{corollary}

\begin{proof}
Equation \eqref{eq: d-1 sl2} follows from Lemma \ref{lem: super diffs}, but we need to prove \eqref{eq: d-1 anticommute}. Observe that
$$
[F,d_{-1}]=F(Fd_{1}-d_{1}F)-(Fd_{1}-d_{1}F)F=F^2d_{1}+d_1F^2-2Fd_1F=0,
$$
so
$$
Fd_1F=\frac{1}{2}(F^2d_{1}+d_1F^2).
$$
Now
$$
d_{-1}d_1+d_1d_{-1}=(Fd_{1}-d_{1}F)d_1+d_1(Fd_{1}-d_{1}F)=Fd_1^2-d_1Fd_1+d_1Fd_1-d_1^2F=0
$$
and
$$
d_{-1}^2=(Fd_{1}-d_{1}F)^2=Fd_1Fd_1-Fd_1^2F-d_1F^2d_1+d_1Fd_1F=
$$
$$
\frac{1}{2}(F^2d_{1}+d_1F^2)d_1-d_1F^2d_1+\frac{1}{2}d_1(F^2d_{1}+d_1F^2)=0.
$$
\end{proof}

\begin{proposition}
\label{prop: blocks}
Assume that a knot $K$ is $d_1$-standard. Then $\HHH(K)$ is filtered by the symmetric blocks of the following form:
$$ 
\begin{tikzcd}
   & 1 \arrow[bend left]{rr}{E} \arrow{dl}{d_{-1}} \arrow{dr}{d_1}&   & 1 \arrow[bend left]{rr}{E}  \arrow{dl}{d_{-1}} \arrow{dr}{d_1}&   & \cdots\arrow[bend left]{rr}{E} \arrow{dl}{d_{-1}} \arrow{dr}{d_1}& & 1  \arrow{dl}{d_{-1}} \arrow{dr}{d_1}& \\
 1    \arrow{dr}{d_1}&   & 2  \arrow{dl}{d_{-1}} \arrow{dr}{d_1}&   & 2 \arrow{dl}{d_{-1}} \arrow{dr}{d_1}& \cdots & 2 \arrow{dl}{d_{-1}} \arrow{dr}{d_1} & & 1 \arrow{dl}{d_{-1}}\\
    & 1 \arrow[bend right]{rr}{E}&   & 1 \arrow[bend right]{rr}{E} &   & \cdots \arrow[bend right]{rr}{E}& & 1 & 
\end{tikzcd}
$$ 
or zigzags of the following form:
$$ 
\begin{tikzcd}
  1      \arrow{dr}{d_1}  \arrow[bend left]{rr}{E}& &\cdots \arrow{dr}{d_1} \arrow{dl}{d_{-1}} \arrow[bend left]{rr}{E}& & 1   \arrow{dl}{d_{-1}}\\
 & 1  & \cdots   & 1 & 
\end{tikzcd}
\quad
\mathrm{or}
\quad
\begin{tikzcd}
 & 1 \arrow{dr}{d_1} \arrow{dl}{d_{-1}}& \cdots   & 1 \arrow{dr}{d_1} \arrow{dl}{d_{-1}}& \\
  1        \arrow[bend right]{rr}{E}& &\cdots   \arrow[bend right]{rr}{E}& & 1 
\end{tikzcd}
$$ 
Furthermore, there is exactly one zigzag of either type.
\end{proposition}

\begin{proof}
We have an action of $\sll(2)\rtimes \langle d_{-1},d_1\rangle$ on $\HHH(K)$ where $d_{-1}$ and $d_1$ square to zero and anticommute by \eqref{eq: d-1 anticommute}. All these operators preserve $\Delta$-grading, so we can focus on one $\Delta$-grading.

Let $U$ be an $\sll(2)$-irreducible summand of $\HHH(K)$ in the top $a$-degree. Define $\widehat{U}=U\otimes \langle 1,d_{-1},d_1,d_{-1}d_1\rangle$. We have a natural $\sll(2)\rtimes \langle d_{-1},d_1\rangle$-invariant map $\varphi:\widehat{U}\to \HHH(K)$, let us describe its possible kernel and image. 

Let $L(n)$ denote the irreducible representation of $\sll(2)$ of highest weight $n$. If $U\simeq L(n)$, then
$$
U\otimes \langle d_{-1},d_1\rangle\simeq L(n)\otimes L(1)\simeq L(n+1)\oplus L(n-1)
$$
as $\sll(2)$ representations, and
$\widehat{U}$ decomposes into irreducible $\sll(2)$ representations as follows: $L(n)$ in top $a$-degree, $L(n+1)\oplus L(n-1)$ in the middle $a$-degree, and $L(n)$ in the bottom $a$-degree. This is precisely the block in the above picture, so if $\Ker(\varphi)=0$ we will see this block in $\HHH(K)$.

Otherwise we have the following options:

1) $\Ker(\varphi)$ contains $L(n+1)$ in the middle $a$-degree and $L(n)$ in the bottom $a$-degree. Then $\Imm(\varphi)$ is isomorphic to a zigzag of the first type.

2) $\Ker(\varphi)$ contains $L(n-1)$ in the middle $a$-degree and $L(n)$ in the bottom $a$-degree. Then $\Imm(\varphi)$ is isomorphic to a zigzag of the second type.

3) $\Ker(\varphi)$ contains both $L(n+1)$ and $L(n-1)$ in the middle $a$-degree, and $L(n)$ in the bottom $a$-degree. Then $\Imm(\varphi)\simeq U$.

4)  $\Ker(\varphi)$ contains only $L(n)$ in the bottom $a$-degree. In this case $\Imm(\varphi)$ contains $L(n)$ on the top $a$-degree and $L(n+1)\oplus L(n-1)$ in the middle $a$-degree.

Now we use the fact that the homology of $d_1$ is one-dimensional. The full block $\widehat{U}$ is acyclic, while the zigzags in cases (1) and (2) contribute one-dimensional homology. In cases (3) and (4) we get $(n+1)$-dimensional homology, which is possible only if $n=0$ and we get special cases of zigzags again.  
\end{proof}

\section{Examples}
\label{sec: examples}
\subsection{Example: $10_{125}$}
\label{sec: small example}

One interesting example is given by the knot $10_{125}$. Its homology is shown in Figure \ref{fig: 10 125}.

\begin{figure}[ht!]
        \includegraphics[width= \textwidth]{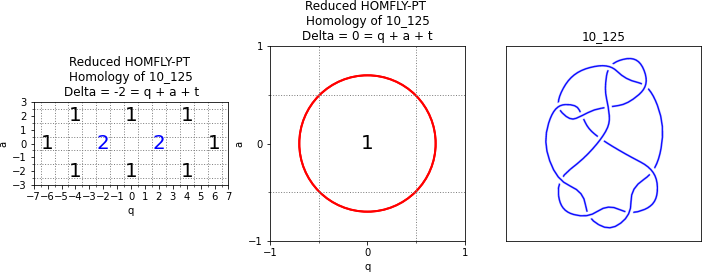}
\caption{Triply graded homology of the knot $10_{125}$ where the $S$-invariant equals 0, while the Rasmussen invariant equals $s=-2$. Furthermore, $s_{x^3-x}=S=0$  while  $s_{x^N-1}=-\frac{1}{N-1}$.}
\label{fig: 10 125}
\end{figure}

Clearly, we have two blocks as in Proposition \ref{prop: blocks} in different $\Delta$-gradings. 

\begin{center}
\begin{tikzpicture}
\draw (-2,2) node {$u$};
\draw (0,2) node {$Eu$};
\draw (2,2) node {$E^2u$};
\draw (-3,1) node {$y$};
\draw (-1,1) node {$(Ey,z)$};
\draw (1,1) node {$(E^2y,Ez)$};
\draw (3,1) node {$E^3y$};
\draw (-2,0) node {$w$};
\draw (0,0) node {$Ew$};
\draw (2,0) node {$E^2w$};
\draw [dotted] (4,-1)--(4,2);
\draw (5,1) node {$m$};
\draw [->](-1.8,2.2)..controls (0,4) and (4,3)..(4.8,1.2);
\draw (3,3) node {$d_2$};
\draw [->] (-2.2,1.8)--(-2.8,1.2);
\draw [->] (-1.8,1.8)--(-1.2,1.2);
\draw [->] (-2.8,0.8)--(-2.2,0.2);
\draw (-2.7,1.7) node {$d_{-1}$};
\draw (-1.3,1.7) node {$d_{1}$};
\draw (-2.7,0.3) node {$d_{1}$};
\draw (0,-1) node {$\Delta=-2$};
\draw (5,-1) node {$\Delta=0$};
\end{tikzpicture}
\end{center}

In $\Delta$-grading $(-2)$ the homology splits into irreducible $\sll(2)$ representations as follows: for $a=2$ we have an irreducible 3-dimensional representation generated by $u$; for $a=0$ we have a direct sum of a 4-dimensional representation generated by $y$ and a 2-dimensional representation generated by $z$; for $a=-2$ we have an irreducible 3-dimensional representation generated by $w$.

We can pin down the generators by requiring that $d_{-1}(u)=y$ and $d_{1}(y)=w$. Then $F(u)=0$, so
$$
F(d_1(u))=[F,d_1](u)=d_{-1}(u)=y,\ d_1^2(u)=0.
$$
On the other hand, $F(Ey)=[F,E]y=-Hy=3y,\ F(z)=0$, and $d_1(Ey)=Ed_1(y)=Ew$. Therefore one can uniquely scale $z$ such that
$$
d_1(u)=\frac{1}{3}Ey+z,\ d_1(z)=-\frac{1}{3}Ew. 
$$

Next, we determine the differential $d_2$. One can check that the total rank of reduced Khovanov homology $H_{\sll(2)}(10_{125})$ is smaller than the rank of $\HHH(10_{125})$, so $d_2$ must be nonzero and (by degree reasons) the only possibility  is $d_2(u)=m$. By Proposition \ref{prop: N large}   we have 
$$
d_N=0,\quad H_{\sll(N)}(10_{125})\simeq \HHH(10_{125})\quad \text{for}\ N\ge 3.
$$

To compute various $s$-invariants, we use Theorem \ref{thm: W s invariant}. Clearly, the homology of $d_1$ is generated by $m$ and $S=0$. To compute $s_2=s_{x^2-x}$, we use the potential $\partial W=x^2-x$ and the differential $d_{x^2-x}=d_2-d_1$. The corresponding spectral sequence has the first differential $d_2$ (which kills $u$ and $m$) and the second differential $-d_1$ with homology spanned by $\frac{1}{6}Ey+z$ supported in $q_{\sll(2)}=-2$. Therefore $s=j_2=-2$ and $s_2(10_{125})=\frac{j_2}{2}=-1$. 

For $N=3$, we have two different deformations of $\sll(3)$ homology.
For $\partial W=x^3-x$, we get the spectral sequence with the first differential $d_3=0$, and the next differential $d_1$ with homology spanned by $m$, so $s_{x^3-x}(10_{125})=0$ as in Theorem \ref{thm: intro sN stabilize}. On the other hand, for $\partial W=(x+1)^3-1=x^3+3x^2+3x$ we get the first differential $d_3=0$ while the second differential $3(d_2+d_1)$ has homology 
$$
\left\langle m,\frac{1}{3}Ey+z\right\rangle /\left(m+\frac{1}{3}Ey+z=0\right)
$$  
Note that $q_{\sll(3)}(m)=0$ while $q_{\sll(3)}\left(\frac{1}{3}Ey+z\right)=-2$. The homology generator $[m]$ has a representative with $q_{\sll(2)}\le -2$, so $H_{\partial W}=\mathcal{F}_{-2}H_{\partial W}$ in the notations of Remark \ref{rem: def j}. Therefore 
$j_{x^3-1}=j_{(x+1)^3-1}=-2$ and $s_{x^3-1}(10_{125})=\frac{-2}{4}=-\frac{1}{2}$.

A similar computation shows $s_{x^N-x}(10_{125})=0$ while $s_{x^N-1}(10_{125})=-\frac{1}{N-1}$ in agreement with \cite[Section 1.2]{LL}.

\subsection{A larger example}
\label{sec: example}

Let us analyze the behaviour of $\sll(N)$ differentials for the knot $11n_{135}$ shown in Figure \ref{homfly-vs-sl2}. We break the computation in several steps:

\begin{figure}[ht!]
    \begin{center}
        \includegraphics[width=0.7\textwidth]{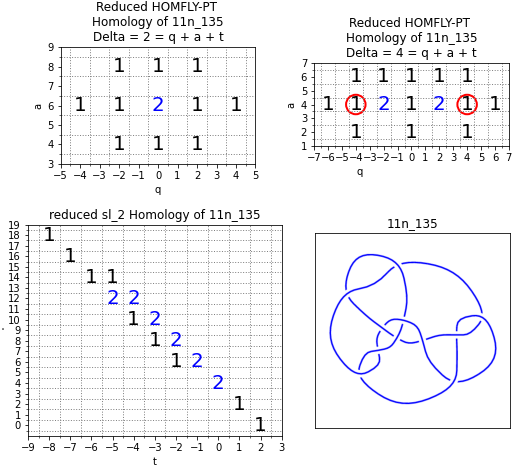}
    \end{center}
    \caption{Here we have both the triply graded homology and $\sll(2)$ homology of the knot $K=11n_{135}$.
    The $S$-invariant equals $-4$ while the Rasmussen $s$-invariant equals $4$.
    }
    \label{homfly-vs-sl2}
\end{figure}

{\bf Step 1:} $N=1$. By Lemma \ref{lem: d1 standard} the knot is $d_1$-standard, so the homology of $d_1$ is one-dimensional in $\Delta=4$ and hence are supported at bidegree $(q,a)=(-4,4)$ marked by the left red circle in Figure \ref{homfly-vs-sl2} and by $n$ in Figure \ref{Fig: 11n135 diff}. In particular, $S(K)=-4$. The differential $d_1$ is acyclic everywhere else, and it is easy to reconstruct it.

{\bf Step 2:} Reconstruct the action of $\sll(2)$: see Figure \ref{Fig: 11n135 diff}.

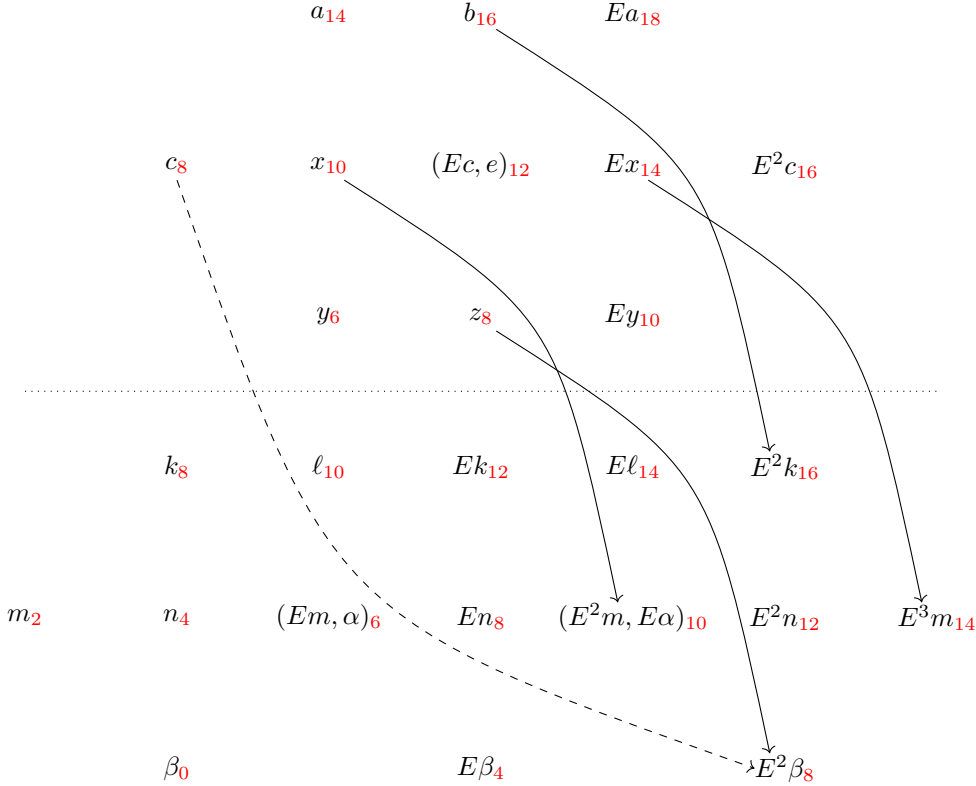
\begin{figure}
\begin{tikzpicture}
\draw (-2,4) node {$a_{\textcolor{red}{14}}$};
\draw (0,4) node {$b_{\textcolor{red}{16}}$};
\draw (2,4) node {$Ea_{\textcolor{red}{18}}$};
\draw (-4,2) node {$c_{\textcolor{red}{8}}$};
\draw (-2,2) node {$x_{\textcolor{red}{10}}$};
\draw (0,2) node {$(Ec,e)_{\textcolor{red}{12}}$};
\draw (2,2) node {$Ex_{\textcolor{red}{14}}$};
\draw (4,2) node {$E^2c_{\textcolor{red}{16}}$};
\draw (-2,0) node {$y_{\textcolor{red}{6}}$};
\draw (0,0) node {$z_{\textcolor{red}{8}}$};
\draw (2,0) node {$Ey_{\textcolor{red}{10}}$};
\draw [dotted] (-6,-1)--(6,-1);
\draw (-4,-2) node {$k_{\textcolor{red}{8}}$};
\draw (-2,-2) node {$\ell_{\textcolor{red}{10}}$};
\draw (0,-2) node {$Ek_{\textcolor{red}{12}}$};
\draw (2,-2) node {$E\ell_{\textcolor{red}{14}}$};
\draw (4,-2) node {$E^2k_{\textcolor{red}{16}}$};
\draw (-6,-4) node {$m_{\textcolor{red}{2}}$};
\draw (-4,-4) node {$n_{\textcolor{red}{4}}$};
\draw (-2,-4) node {$(Em,\alpha)_{\textcolor{red}{6}}$};
\draw (0,-4) node {$En_{\textcolor{red}{8}}$};
\draw (2,-4) node {$(E^2m,E\alpha)_{\textcolor{red}{10}}$};
\draw (4,-4) node {$E^2n_{\textcolor{red}{12}}$};
\draw (6,-4) node {$E^3m_{\textcolor{red}{14}}$};
\draw (-4,-6) node {$\beta_{\textcolor{red}{0}}$};
\draw (0,-6) node {$E\beta_{\textcolor{red}{4}}$};
\draw (4,-6) node {$E^2\beta_{\textcolor{red}{8}}$};

\draw [->] (0.2,3.8)..controls (3,2) and (3,2)..(3.8,-1.8);
\draw [->] (2.2,1.8)..controls (5,0) and (5,0)..(5.8,-3.8);
\draw [->] (-1.8,1.8)..controls (1,0) and (1,0)..(1.8,-3.8);
\draw [->] (0.2,-0.2)..controls (3,-2) and (3,-2)..(3.8,-5.8);

\draw [->,dashed] (-4,1.8)..controls (-2,-4) and (-2,-4)..(3.6,-6);
\end{tikzpicture}
\caption{Differentials in the $\sll(2)$ spectral sequence for $11n_{135}$. The generators in $\Delta$-grading 2 are above the line, and the generators in $\Delta$-grading 4 are below the line.}
\label{Fig: 11n135 diff}
\end{figure}

In particular, we see that $d_1(b)=Ex$ and $d_1(x)=z$ and, up to a choice of scalars, we get 
$$
d_1(m)=\beta,\ d_1(k)=Em+\alpha,\ d_1(\alpha)=-E\beta.
$$
 
{\bf Step 3:} $N=2$. Let us reconstruct the action of $d_2$. Recall that $q_{\sll(2)}=q+2a$, these degrees are marked in red in Figure \ref{Fig: 11n135 diff}, and $d_2$ preserves $q_{\sll(2)}$ while decreasing $a$-degree by 2 and increasing $\Delta$-grading by 2. Potentially, there are also higher differentials $d_2^{(i)}$ which  preserve $q_{\sll(2)}$ while decreasing $a$-degree by $2i$ and increasing $\Delta$-grading by 2.  

In degree $q_{\sll(2)}=16$ we have three generators in $\HHH(K)$ and only one generator in $H_{\sll(2)}(K)$.
Therefore there should be exactly one cancellation and by looking at $a$-degrees we conclude that there is no room for higher differentials and $d_2(b)=E^2k$ (up to a nonzero coefficient which we ignore).   

Now $[d_1,d_2]=0$, so 
$$
d_2(Ex)=d_2(d_1(b))=d_1(d_2(b))=d_1(E^2k)=E^3m,
$$ 
This is the only cancellation in degree 14, and this computation excludes another potential cancellation or higher differential. 
There are no cancellations in degree 12.
Since $[d_2,E]=0$, we get 
$$
E(d_2(x))=d_2(Fx)=E^3m,
$$
so $d_2(x)\neq 0$, this is the only cancellation in degree 10.

Finally, let us understand the cancellations in degree 8. If $d_2(c)\neq 0$ then $d_2(Ec)=E(d_2(c))\neq 0$, and there is a cancellation in degree 12. Contradiction, hence $d_2(c)=0$. We have two cases:
$$
\begin{cases}
d_2(z)=E^2\beta,\ d_2^{(2)}=0\\
d_2(z)=0,\ d_2^{(2)}=E^2\beta.
\end{cases}
$$
In the first case (shown in Figure \ref{Fig: 11n135 diff}) we have we have $d_1(d_2(x))\neq 0$ while in the second case (shown by the dashed arrow) we have $d_1(d_2(x))=0$, so $d_2(x)$ is proportional to $d_1(k)$. 

{\bf Step 4:} $N\ge 3$. By Proposition \ref{prop: N large} all differentials $d_N^{(i)}$ vanish for $N\ge 3$, and $H_{\sll(N)}(K)=\HHH(K)$ up to regrading.

{\bf Step 5:} In either case above, the generator $n$ survives in $\sll(2)$ spectral sequence, so $s_2=\frac{1}{2}(-4+2\cdot 4)=2$. Similarly, $s_{\partial W}=\frac{1}{2(N-1)}(-4+N\cdot 4)=2$ for all $N$ and arbitrary potential $\partial W$ of degree $N$. The Rasmussen $s$-invariant equals $s=2s_2=4$, in agreement with \cite{KnotInfo}.  

\subsection{Example: $T(4,5)$}
\label{sec: T45}

In this subsection we consider the 15-crossing torus knot $T(4,5)$ which is not covered by the data in \cite{NS}.  On the other hand, Hogancamp in \cite{Hog} proved a conjecture of the second author \cite{Gor} relating the triply graded homology of torus knots $T(n,n+1)$ to so-called $q,t$-Catalan numbers. 
In particular, $\HHH(T(4,5))$ agrees with its conjectural description in \cite[Section 3.4]{Gor} and is presented in Figure \ref{fig: T45}.
The total dimension of $\HHH(T(4,5))$ equals 45.

We can give a more precise description of the differentials on $\HHH(T(4,5))$ using the above results. First, $T(4,5)$ is parity and by Corollary \ref{cor: dN parity} $d_N^{(i)}$ vanish for $i$ even. It is easy to check that by degree reasons $d_N^{(i)}=0$ for $i\ge 3$, hence the Rasmussen spectral sequences from Proposition \ref{prop: differentials} collapse after the first differential $d_N$ for all $N$.

In particular, $T(4,5)$ is $d_1$-standard and $\HHH(T(4,5))$ decomposes into blocks as in Proposition \ref{prop: blocks}, as Figure \ref{fig: T45} clearly demonstrates. In fact, there is only one block in each $\Delta$-grading in this case. The action of $\sll(2)$ is clear from the blocks.
The homology of $d_1$ is one-dimensional, it is marked by the red circle.

The reduced $\sll(2)$ homology has rank 9 and is shown in yellow in Figure \ref{fig: T45} (a half-colored box corresponds to one-dimensional $H_{\sll(2)}$ and two-dimensional $\HHH$ in a given degree).  

The reduced $\sll(3)$ homology has rank 23 and is shown in white and yellow, in agreement with \cite{FoamHo}. The differential $d_3$ increases $\Delta$ by 4, and cancels the blue regions for $\Delta=6,10$ as well as the green regions for $\Delta=8,12$. 

Finally, we can study the operators $d_{a|b}$ from Lemma \ref{lem: super diffs}.  
Let $X$ denote the generator at the very top of Figure \ref{fig: T45}, with $(q,a,\Delta)=(0,18,6)$. Since $F(X)=0$, we observe that:
\begin{itemize}
\item There is a length 2 $\sll(2)$ chain containing $d_1(X)$ and $Fd_1(X)=d_{0|1}(X)=d_{-1}(X)$.
\item There is a length 3 $\sll(2)$ chain containing $d_2(X)$, $Fd_2(X)=d_{1|1}(X)$ and $\frac{1}{2}F^2d_2(X)=d_{0|2}(X)$.
\item There is a length 4 $\sll(2)$ chain containing $d_3(X)$, $Fd_3(X)=d_{2|1}(X)$, $\frac{1}{2}F^2d_3(X)=d_{1|2}(X)$ and
$\frac{1}{6}F^3d_3(X)=d_{0|3}(X)$.
\end{itemize}
In particular, all of the operators $d_{a|b}$ for $a,b\le 3$ are nontrivial on $\HHH(T(4,5))$. 

We expect that for a more general torus knot $T(n,m)$ the operators $d_{a|b}$ are nontrivial for $a,b\le \min(n,m)-1$ and plan to investigate them in more detail in a future work.

\begin{figure}

\begin{tikzpicture}[scale=0.5]
\filldraw[color=blue!20] (-1,11)--(-1,13)--(-3,13)--(-3,15)--(-1,15)--(-1,17)--(1,17)--(1,15)--(3,15)--(3,13)--(1,13)--(1,11)--(-1,11);

\filldraw[color=white] (-1,13)--(-1,15)--(1,15)--(1,13)--(-1,13);

\filldraw[color=blue!20] (5,-1)--(5,1)--(4,1)--(4,3)--(5,3)--(5,5)--(7,5)--(7,3)--(9,3)--(9,1)--(7,1)--(7,-1)--(5,-1);

\filldraw[color=white] (5,1)--(5,3)--(7,3)--(7,1)--(5,1);

\filldraw [color=green!20] (-7,7)--(-7,9)--(-5,9)--(-5,11)--(5,11)--(5,9)--(7,9)--(7,7)--(3,7)--(3,8)--(-3,8)--(-3,7)--(-7,7);

\filldraw[color=white] (-5,7)--(-5,9)--(-3,9)--(-3,7)--(-5,7);

\filldraw[color=white] (-3,9)--(-3,11)--(-1,11)--(-1,9)--(-3,9);

\filldraw[color=white] (-1,7)--(-1,9)--(1,9)--(1,7)--(-1,7);

\filldraw[color=white] (1,9)--(1,11)--(3,11)--(3,9)--(1,9);

\filldraw[color=white] (3,7)--(3,9)--(5,9)--(5,7)--(3,7);

\filldraw [color=green!20] (-1,-5)--(-1,-3)--(1,-3)--(1,-1)--(11,-1)--(11,-3)--(13,-3)--(13,-5)--(-7,-5);

\filldraw[color=white] (1,-5)--(1,-3)--(3,-3)--(3,-5)--(1,-5);

\filldraw[color=white] (3,-3)--(3,-1)--(5,-1)--(5,-3)--(3,-3);

\filldraw[color=white] (5,-5)--(5,-3)--(7,-3)--(7,-5)--(5,-5);

\filldraw[color=white] (7,-3)--(7,-1)--(9,-1)--(9,-3)--(7,-3);

\filldraw[color=white] (9,-5)--(9,-3)--(11,-3)--(11,-5)--(9,-5);

\filldraw [color=yellow!20] (-13,-5)--(-13,-3)--(-11,-3)--(-11,-1)--(-5,-1)--(-5,-3)--(-7,-3)--(-7,-5)--(-13,-5);

\filldraw[color=white] (-11,-5)--(-11,-3)--(-9,-3)--(-9,-5)--(-11,-5);

\filldraw[color=white] (-9,-3)--(-9,-1)--(-7,-1)--(-7,-3)--(-9,-3);

\filldraw [color=yellow!20] (-7,-1)--(-7,1)--(-5,1)--(-5,-1)--(-7,-1);

\filldraw [color=yellow!20] (-5,1)--(-5,2)--(-3,2)--(-3,1)--(-5,1);

\filldraw [color=yellow!20] (-5,5)--(-5,7)--(-3,7)--(-3,8)--(-1,8)--(-1,7)--(1,7)--(1,5)--(-5,5);

\filldraw[color=white] (-3,5)--(-3,7)--(-1,7)--(-1,5)--(-3,5);

\draw (0,16) node {1};
\draw (-2,14) node {1};
\draw (2,14) node {1};
\draw (0,12) node {1};
\draw [line width=2] (-14,11)--(14,11);
\draw (-4,10) node {1};
\draw (0,10) node {1};
\draw (4,10) node {1};
\draw (-6,8) node {1};
\draw (-2,8) node {2};
\draw (2,8) node {2};
\draw (6,8) node {1};
\draw (-4,6) node {1};
\draw (0,6) node {1};
\draw (4,6) node {1};
\draw [line width=2] (-14,5)--(14,5);
\draw (-6,4) node {1};
\draw (-2,4) node {1};
\draw (2,4) node {1};
\draw (6,4) node {1};
\draw (-8,2) node {1};
\draw (-4,2) node {2};
\draw (0,2) node {2};
\draw (4,2) node {2};
\draw (8,2) node {1};
\draw (-6,0) node {1};
\draw (-2,0) node {1};
\draw (2,0) node {1};
\draw (6,0) node {1};
\draw [line width=2] (-14,-1)--(14,-1);
\draw (-10,-2) node {1};
\draw (-6,-2) node {1};
\draw (-2,-2) node {1};
\draw (2,-2) node {1};
\draw (6,-2) node {1};
\draw (10,-2) node {1};
\draw (-12,-4) node {1};
\draw (-8,-4) node {1};
\draw (-4,-4) node {1};
\draw (0,-4) node {1};
\draw (4,-4) node {1};
\draw (8,-4) node {1};
\draw (12,-4) node {1};

\draw (-13,-5)--(-13,17);
\draw (-11,-5)--(-11,17);
\draw (-9,-5)--(-9,17);
\draw (-7,-5)--(-7,17);
\draw (-5,-5)--(-5,17);
\draw (-3,-5)--(-3,17);
\draw (-1,-5)--(-1,17);
\draw (1,-5)--(1,17);
\draw (3,-5)--(3,17);
\draw (5,-5)--(5,17);
\draw (7,-5)--(7,17);
\draw (9,-5)--(9,17);
\draw (11,-5)--(11,17);
\draw (13,-5)--(13,17);

\draw (-13,-5)--(13,-5);
\draw (-13,-3)--(13,-3);
\draw (-13,-5)--(13,-5);
\draw (-13,1)--(13,1);
\draw (-13,3)--(13,3);
\draw (-13,7)--(13,7);
\draw (-13,9)--(13,9);
\draw (-13,13)--(13,13);
\draw (-13,15)--(13,15);
\draw (-13,17)--(13,17);

\draw [->] (13,-5)--(14,-5);
\draw [->] (-13,17)--(-13,18);
\draw (14,-5.5) node {$q$};
\draw (-13.5,17) node {$a$};

\draw (-12,-5.5) node {-12};
\draw (-10,-5.5) node {-10};
\draw (-8,-5.5) node {-8};
\draw (-6,-5.5) node {-6};
\draw (-4,-5.5) node {-4};
\draw (-2,-5.5) node {-2};
\draw (0,-5.5) node {0};
\draw (2,-5.5) node {2};
\draw (4,-5.5) node {4};
\draw (6,-5.5) node {6};
\draw (8,-5.5) node {8};
\draw (10,-5.5) node {10};
\draw (12,-5.5) node {12};

\draw (-13.5,-4) node {12};
\draw (-13.5,-2) node {14};
\draw (-13.5,0) node {12};
\draw (-13.5,2) node {14};
\draw (-13.5,4) node {16};
\draw (-13.5,6) node {12};
\draw (-13.5,8) node {14};
\draw (-13.5,10) node {16};
\draw (-13.5,12) node {14};
\draw (-13.5,14) node {16};
\draw (-13.5,16) node {18};

\draw (-16,-3) node {$\Delta=12$};
\draw (-16,2) node {$\Delta=10$};
\draw (-16,8) node {$\Delta=8$};
\draw (-16,14) node {$\Delta=6$};

\draw [color=red,line width=2] (-12,-4) circle (1);

\draw [->,dotted] (1.5,17.5)--(0.5,16.5);
\draw (2,18) node {$X$};
\draw [->,dotted] (3.5,15.5)--(2.5,14.5);
\draw (4,16) node {$d_1(X)$};
\draw [->,dotted] (5.5,11.5)--(4.5,10.5);
\draw (6,12) node {$d_2(X)$};
\draw [->,dotted] (7.5,5.5)--(6.5,4.5);
\draw (8,6) node {$d_3(X)$};

\end{tikzpicture}
\caption{The reduced $\mathfrak{sl}(2)$ homology  (yellow) and the reduced $\mathfrak{sl}(3)$ homology  (yellow and white) of $T(4,5)$. These have total ranks 9 and 23 respectively. The differential $d_3$ cancels blue and green regions in pairs.}
\label{fig: T45}
\end{figure}




\section{Appendix: Potentially $d_1$-nonstandard knots}
\label{sec: appendix}

The figures below illustrate the exceptional cases in the proof of Proposition \ref{prop: d1 exceptions}. 
The knots $10_{136}$, $11n_{12}$, $11n_{20}$, and $11n_{79}$ have no room for the differential $d_1^{(2)}$ because there are no nonzero homology groups in the necessary gradings (see figures \ref{10_136-fig}, \ref{11n_12-fig},\ref{11n_20-fig}, and \ref{11n_79-fig}). 
For the remaining exceptional cases, we indicate the location of the potential higher differential $d_1^{(2)}$ in the $\sll(1)$ spectral sequence and conclude that it vanishes, otherwise the dimension of the $E_{\infty}$ page is at least 2. 

\begin{figure}[ht!]
\includegraphics[width=\textwidth]{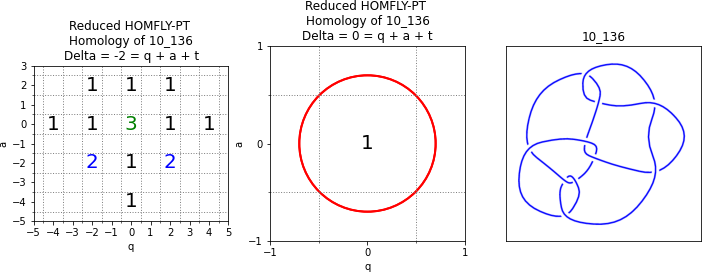}
\caption{$10_{136}$: For there to be a nonzero $d_1^{(2)}$ there would have to be nonzero homology group in $(q,a,\Delta) = (4,-4,-2)$.}
\label{10_136-fig}
\end{figure}

\begin{figure}[ht!]
\includegraphics[width=\textwidth]{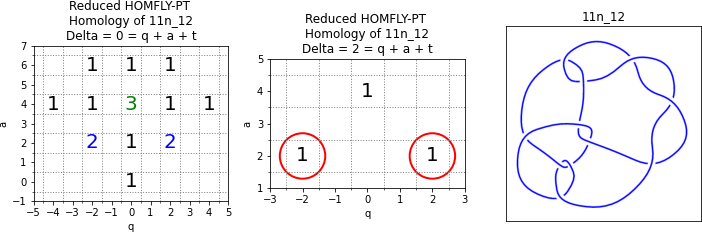}
\caption{$11n_{12}$: For there to be a nonzero $d_1^{(2)}$ there would have to be nonzero homology group in $(q,a,\Delta) = (4,0,0), (2,-2,0)$ or $(6,-2,0)$.}
\label{11n_12-fig}
\end{figure}

\begin{figure}[ht!]
\includegraphics[width=\textwidth]{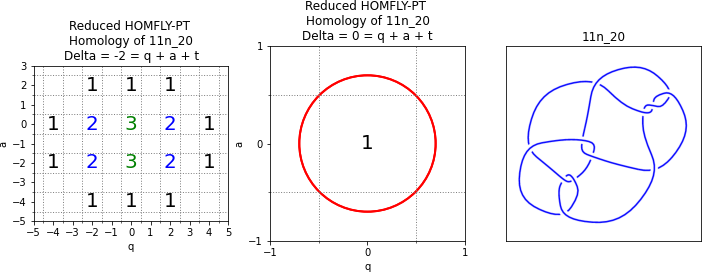}
\caption{$11n_{20}$: For there to be a nonzero $d_1^{(2)}$ there would have to be nonzero homology group in $(q,a,\Delta) = (4,-4,-2)$.}
\label{11n_20-fig}
\end{figure}

\begin{figure}[ht!]
\includegraphics[width=\textwidth]{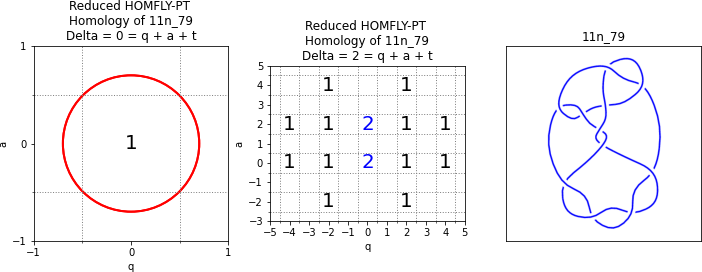}
\caption{$11n_{79}$: For there to be a nonzero $d_1^{(2)}$ there would have to be nonzero homology group in $(q,a,\Delta) = (-4,4,2)$.}
\label{11n_79-fig}
\end{figure}

\begin{figure}[ht!]
\includegraphics[width=\textwidth]{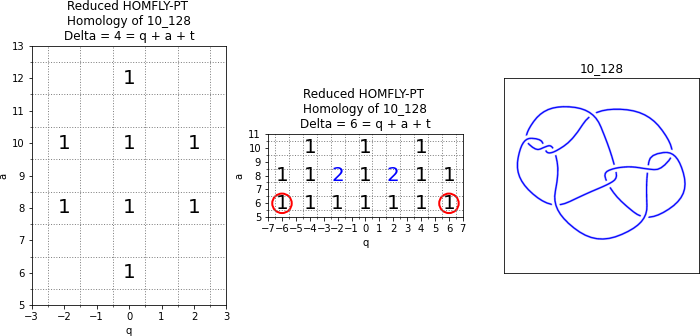}
\caption{$10_{128}$: There is a potential differential $d_1^{(2)}$ from $(q,a)=(-4,10)$ to $(0,6)$ which must vanish. The $S$-invariant equals $-6$.}
\end{figure}

\begin{figure}[ht!]
\includegraphics[width=\textwidth]{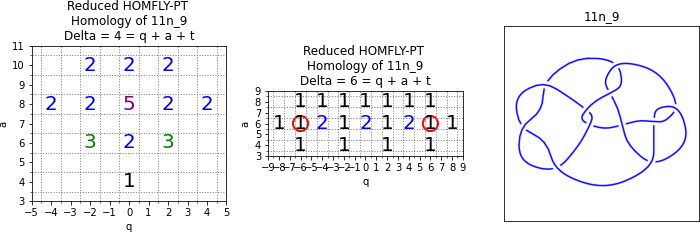}
\caption{$11n_{9}$: There is a potential differential $d_1^{(2)}$ from $(q,a)=(-4,8)$ to $(0,4)$ which must vanish. The $S$-invariant equals $-6$.}
\end{figure}

\begin{figure}
\includegraphics[width=\textwidth]{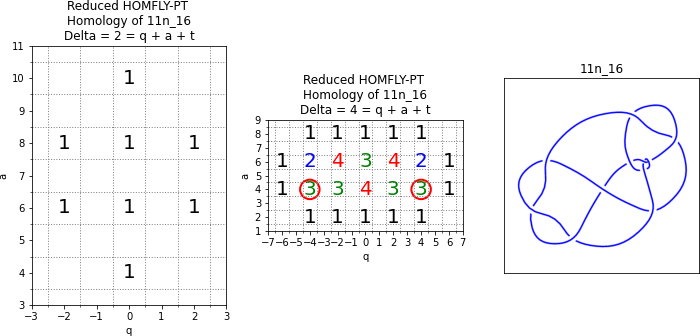}
\caption{$11n_{16}$: There is a potential differential $d_1^{(2)}$ from $(q,a)=(-4,8)$ to $(0,4)$ which must vanish. The $S$-invariant equals $-4$.}
\end{figure}

\begin{figure}
\includegraphics[width=\textwidth]{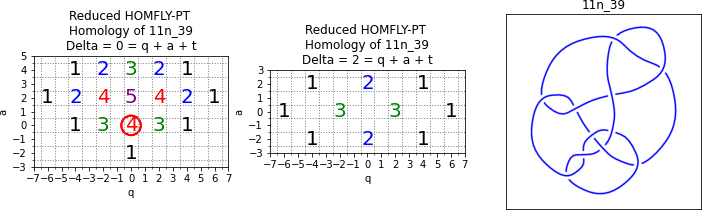}
\caption{$11n_{39}$: There is a potential differential $d_1^{(2)}$ from $(q,a)=(-4,2)$ to $(0,-2)$ which must vanish. The $S$-invariant equals $0$.}
\end{figure}

\begin{figure}
\includegraphics[width=\textwidth]{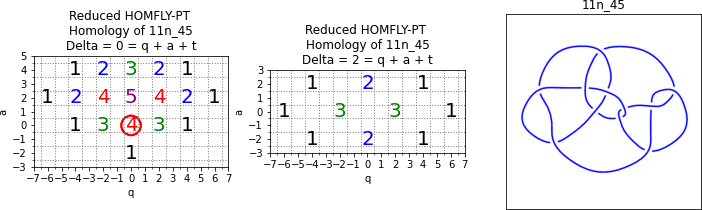}
\caption{$11n_{45}$: There is a potential differential $d_1^{(2)}$ from $(q,a)=(-4,2)$ to $(0,-2)$ which must vanish. The $S$-invariant equals $0$.}
\end{figure}

\begin{figure}
\includegraphics[width=\textwidth]{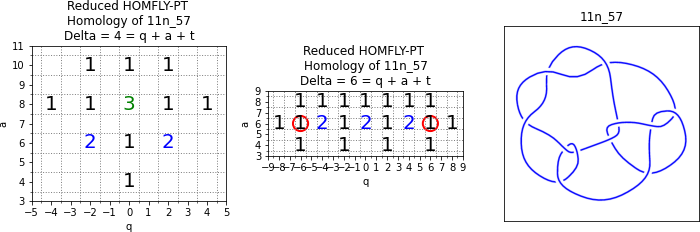}
\caption{$11n_{57}$: There is a potential differential $d_1^{(2)}$ from $(q,a)=(-4,8)$ to $(0,4)$ which must vanish. The $S$-invariant equals $-6$.}
\end{figure}

\begin{figure}
\includegraphics[width=\textwidth]{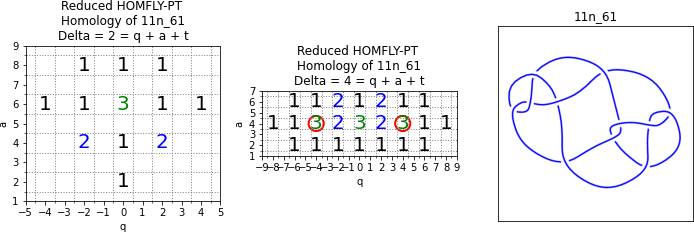}
\caption{$11n_{61}$: There is a potential differential $d_1^{(2)}$ from $(q,a)=(-4,6)$ to $(0,2)$ which must vanish. The $S$-invariant equals $-4$.}
\end{figure}

\begin{figure}
\includegraphics[width=\textwidth]{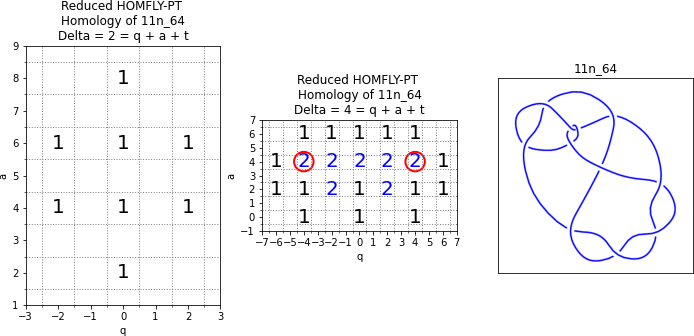}
\caption{$11n_{64}$: There is a potential differential $d_1^{(2)}$ from $(q,a)=(-4,6)$ to $(0,2)$ which must vanish. The $S$-invariant equals $-4$.}
\end{figure}

\begin{figure}
\includegraphics[width=\textwidth]{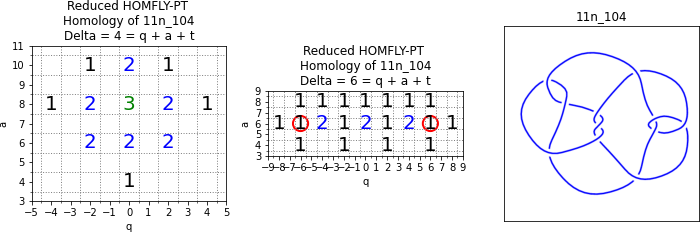}
\caption{$11n_{104}$: There is a potential differential $d_1^{(2)}$ from $(q,a)=(-4,8)$ to $(0,4)$ which must vanish. The $S$-invariant equals $-6$.}
\end{figure}

\begin{figure}
\includegraphics[width=\textwidth]{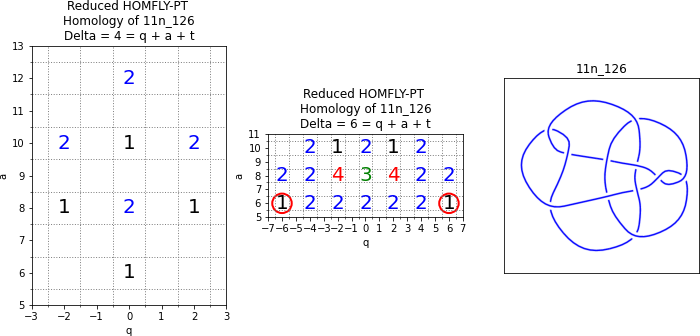}
\caption{$11n_{126}$: There is a potential differential $d_1^{(2)}$ from $(q,a)=(-4,10)$ to $(0,6)$ which must vanish. The $S$-invariant equals $-6$.}
\end{figure}

\begin{figure}
\includegraphics[width=\textwidth]{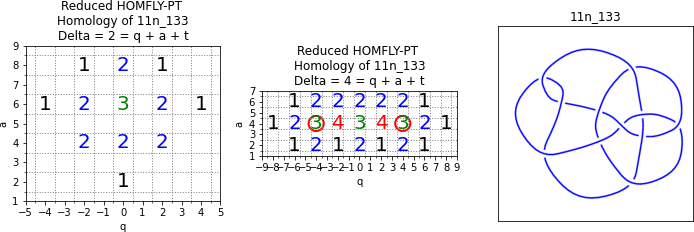}
\caption{$11n_{133}$: There is a potential differential $d_1^{(2)}$ from $(q,a)=(-4,6)$ to $(0,2)$ which must vanish. The $S$-invariant equals $-4$.}
\end{figure}

\begin{figure}
\includegraphics[width=\textwidth]{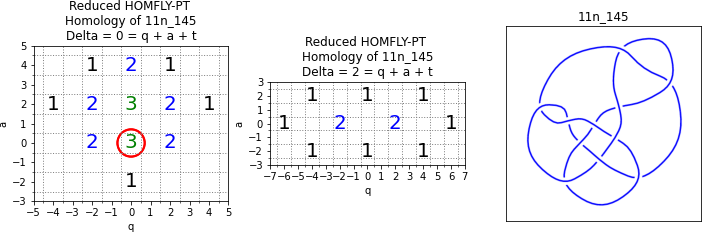}
\caption{$11n_{145}$: There is a potential differential $d_1^{(2)}$ from $(q,a)=(-4,2)$ to $(0,-2)$ which must vanish. The $S$-invariant equals $0$.}
\end{figure}

\begin{figure}
\includegraphics[width=\textwidth]{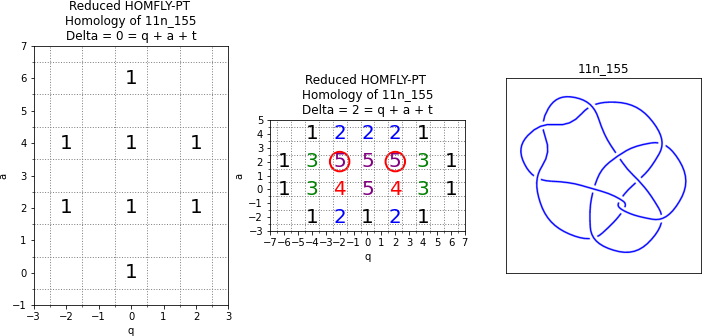}
\caption{$11n_{155}$: There is a potential differential $d_1^{(2)}$ from $(q,a)=(-4,4)$ to $(0,0)$ which must vanish. The $S$-invariant equals $-2$.}
\end{figure}

\end{document}